\documentclass[11pt]{amsart}

\usepackage{amsmath}
\usepackage{amssymb}
\usepackage{graphicx}  
\usepackage{color}
\usepackage{faktor}
\usepackage{comment}

\title[Basmajian--type inequalities for maximal representations]{Basmajian--type inequalities\\for maximal representations}
\author{Federica Fanoni}
\address[Federica Fanoni]{Mathematical Institute, University of Heidelberg, Heidelberg, Germany}
\email{federica.fanoni@gmail.com}
\author{Maria Beatrice Pozzetti}
\address[Maria Beatrice Pozzetti]{Mathematical Institute, University of Heidelberg, Heidelberg, Germany}
\email{pozzetti@mathi.uni-heidelberg.de}
\date{\today}

\usepackage[margin=1.2in]{geometry}
\usepackage{bm}
\usepackage[percent]{overpic}
\usepackage{float}
\usepackage{enumerate}

\DeclareMathOperator{\Sp}{Sp}
\DeclareMathOperator{\R}{\mathbb{R}}
\DeclareMathOperator{\B}{\mathbb{B}}
\DeclareMathOperator{\N}{\mathbb{N}}
\DeclareMathOperator{\hyp}{\mathbb{H}^2}
\DeclareMathOperator{\ort}{\mathcal{O}}

\DeclareMathOperator{\arccoth}{arccoth}
\DeclareMathOperator{\lm}{\mu_{\mbox{\tiny Leb}}}
\DeclareMathOperator{\im}{Im}
\DeclareMathOperator{\re}{Re}

\newcommand{\G}{\Gamma}
\newcommand{\Yy}{\mathcal Y}
\renewcommand{\L}{\Lambda}
\newcommand{\Hom}{{\rm Hom}}
\renewcommand{\S}{\Sigma}
\newcommand{\PSL}{{\rm PSL}}
\renewcommand{\R}{\mathbf R}
\newcommand{\PSp}{{\rm PSp}}

\renewcommand{\O}{{\rm O}}
\newcommand{\C}{\mathbf{C}}
\newcommand{\<}{\langle}
\renewcommand{\H}{\mathbb H}
\renewcommand{\>}{\rangle}

\newcommand{\wt}{\widetilde}
\newcommand{\bsm}{\left(\begin{smallmatrix}}
\newcommand{\esm}{\end{smallmatrix}\right)}                   
\newcommand{\bpm}{\begin{pmatrix}}
\newcommand{\epm}{\end{pmatrix}}
\newcommand{\GL}{{\rm GL}}
\newcommand{\Id}{{\rm Id}}
\newcommand{\Ll}{\mathcal L}

\newcommand{\Sym}{{\rm Sym}}
\newcommand{\ov}{\overline}
\renewcommand{\l}{\lambda}
\newcommand{\SL}{{\rm SL}}
\renewcommand{\d}{\delta}
\newcommand{\g}{\gamma}
\newcommand{\aw}{\bar{\mathfrak{a}}^+}
\DeclareMathOperator{\diag}{diag}
\newcommand{\dC}{{\rm d}^{\aw}} 
\newcommand{\ellC}{{\ell}^{\aw}}
\newcommand{\dF}{{\rm d}^{F}} 
\newcommand{\ellF}{{\ell}^{F}}
\newcommand{\dR}{{\rm d}^{R}} 
\newcommand{\ellR}{{\ell}^{R}}
\newcommand{\calX}{\mathcal X}
\newcommand{\Stab}{{\rm Stab}}
\newcommand{\quot}[2]{{\left. \raisebox{1.5px}{$#1$}\middle/ \raisebox{-5px}{$#2$}\right.}}
\newcommand{\Xx}{\calX}
\newcommand{\s}{\sigma}

\newcommand{\dGL}{{\rm d}_{\GL}}
\newcommand{\dSL}{{\rm d}_{\SL}}
\newcommand{\ev}{{\rm ev}}

\newtheorem{theor}{Theorem}[section]
\newtheorem{cor}[theor]{Corollary}

\newtheorem{lem}[theor]{Lemma}
\newtheorem{prop}[theor]{Proposition}

\newtheorem{theorintro}{Theorem}

\newtheorem{lem?}[theor]{Lemma?}
\newtheorem{definition}[theor]{Definition}
\newtheorem*{Bas}{Basmajian's identity}

\theoremstyle{definition}
\newtheorem{Remark}[theor]{Remark}
\newtheorem{remark}[theor]{Remark}

\theoremstyle{definition}

\numberwithin{equation}{section}

\begin{document}
\begin{abstract}
For suitable metrics on the locally symmetric space associated to a maximal representation, we prove  inequalities between the length of the boundary and the lengths of orthogeodesics that generalize the classical Basmajian's identity from Teichm\"uller theory. Any equality characterizes diagonal embeddings.  \end{abstract}
\maketitle
\section{Introduction}
In the last few decades, many authors have proved elegant identities over moduli spaces of hyperbolic surfaces. One of the first identities was proven by Basmajian in \cite{basmajian}. Even though his work applies more generally to hyperbolic $n$-manifolds for any $n\geq 2$, his most celebrated result is the following:
\begin{Bas}[\cite{basmajian}]
Let $\Sigma$ be a hyperbolic surface with nonempty geodesic boundary $\partial \Sigma$, and let $\ort^{\hyp}_{\Sigma}$ denote the set of unoriented orthogeodesics of $\Sigma$ (i.e.\ the set of geodesics with endpoints on the boundary and orthogonal to it). Then $$\ell(\partial\Sigma)=4\sum_{\alpha\in\ort^{\hyp}_\Sigma}\log \coth \frac{\ell(\alpha)}{2}.$$
\end{Bas}
Another identity involving orthogeodesics is due to Bridgeman and Kahn \cite{bk}, who show that the volume of a hyperbolic $n$-manifold with geodesic boundary can be computed in terms of the length of orthogeodesics. Vlamis and Yarmola generalized this result to a larger class of manifolds in \cite{VY_BKid}.

Other well-known identities are McShane's identity \cite{mcshane} and its generalization by Mirzakhani \cite{mirzakhani}, Luo--Tan's identity \cite{lt} and Bridgeman's identity \cite{bridgeman}. Note that, besides the intrinsic interest of these results, some also have important applications: for instance, Mirzakhani used in \cite{mirzakhani} the generalization of McShane's identity to give a recursive formula for the Weil-Petersson volume of moduli spaces of surfaces with boundary. 

Our main goal is to prove inequalities which generalize Basmajian's identity to the context of maximal representations and such that equality is attained exactly when the maximal representation is a diagonal embedding of a hyperbolization. A tool to prove these inequalities is a result of independent interest: an identity involving cross-ratios (in the sense of Labourie), which should be thought of as the higher rank analogue of the fact that the limit set of the fundamental group of a surface with boundary has measure zero.
\subsection{Maximal representations as higher Teichm\"uller theory} 
{ Maximal representations form a class of discrete and injective representations of fundamental groups of surfaces\footnote{We will only consider oriented surfaces with negative Euler characteristic.} into the symplectic group. They consist of the representations maximizing the \emph{Toledo invariant}, a notion of volume defined using bounded cohomology. A fundamental result of Burger, Iozzi and Wienhard \cite[Theorem 8]{biw}, characterizes them as the representations into $\Sp(2n,\R)$ admitting an equivariant ``well-behaved'' \emph{boundary map} $\phi:\partial \pi_1(\Sigma)\to \Ll(\R^{2n})$ from the boundary of the fundamental group of the surface to the Lagrangians of $\R^{2n}$
(see \cite{biw} or Section \ref{maxreps} for the precise definitions).}

These representations are a generalization of Teichm\"uller space: it was proven by Goldman in \cite{goldman} that Teichm\"uller space can be interpreted as the parameter space of representations of the fundamental group of a surface into $\PSL(2,\R)=\PSp(2,\R)$ maximizing the Euler class, of which the Toledo invariant is a higher rank generalization. In particular, maximal representations into $\PSp(2,\R)$ are holonomies of hyperbolizations and hence correspond to the Teichm\"uller space. The study of common patterns of maximal representations and other special representations, most notably Hitchin representations and positive representations, is often referred to as \emph{higher Teichm\"uller theory}.

A recent trend in higher Teichm\"uller theory is to see which results of hyperbolic geometry can be generalized in the context of representations of surface groups. For instance, a classical and extremely useful result about hyperbolic surfaces is the \emph{collar lemma}. First proven by Keen \cite{keen}, it states the existence, for a simple closed geodesic $\gamma$, of a neighborhood which is an embedded cylinder of width depending only on $\ell(\gamma)$, diverging when $\ell(\gamma)$ shrinks to zero. In particular, one can deduce a lower bound in terms of $\ell(\gamma)$ to the length of any simple closed geodesic intersecting $\gamma$. This corollary has been generalized to the Hitchin component by Lee and Zhang \cite{lz} and to maximal representation by Burger and the second author \cite{bp}.

Also identities in the higher Teichm\"uller setting have attracted attention in recent years: Labourie and McShane generalized McShane-Mirzakhani's identities to arbitrary cross-ratios and in particular to Hitchin representations in \cite{LMS}, while in \cite{VY} Vlamis and Yarmola obtained a generalization of Basmajian's identity to the Hitchin component.
\subsection{Main results}
A maximal representation $\rho:\pi_1(\Sigma)\rightarrow \Sp(2n,\R)$ induces an action of $\Gamma=\pi_1(\Sigma)$ on the symmetric space $\calX$ associated to the symplectic group. The action is properly discontinuous and therefore the locally symmetric space $\rho(\G)\backslash \calX$ is a smooth manifold.  
As we will describe in Section \ref{metrics}, we consider three $\Sp(2n,\R)$-invariant distances on $\calX$: the determinant Finsler distance $\dF$, the Riemannian distance $\dR$ and a Weyl chamber valued distance $\dC$. The latter has values in a Weyl chamber $\aw$, a specific subset of $\R^n$ that parametrizes the orbits of $\Sp(2n,\R)$ on the tangent bundle ${\rm T}\calX$. 
We denote by $\ellR$, $\ellF$ and $\ellC$ the length of paths in $\calX$ computed using $\dR$, $\dF$ and $\dC$ respectively. If $\gamma$ is an element of the fundamental group, its length (with respect to each metric) will be its translation length: the infimum over all points $X\in\calX$ of the distance between $X$ and its translate $\rho(\gamma)\cdot X$.

The first goal of the paper is to find the suitable  generalization of classical orthogeodesics to the context of maximal representations. In the classical setup, an orthogeodesic between two boundary components of a hyperbolic surface lifts in $\hyp$ to a geodesic segment orthogonal to two lifts of the boundary components in the hyperbolic plane. It is pointed out in \cite{bp} that a good generalization of geodesics in the setting of maximal representations is given by the so-called $\R$-tubes, parallel sets of specific singular geodesics (see Section \ref{Rtubes}). In fact, to each $\gamma\in\Gamma$ we can naturally associate an $\R$-tube $\Yy_\gamma$, which is a $\rho(\gamma)$-invariant subspace of $\calX$. In Section \ref{sec:orthogeodesics}, we observe that given any pair of primitive peripheral elements $\gamma$ and $\delta$ in $\Gamma$, there exists a unique $\R$-tube $\Yy_\alpha$ (sometimes simply denoted by $\alpha$) which is orthogonal to both $\Yy_\gamma$ and $\Yy_\delta$; moreover, this tube meets each subspace in a point. We call $\Yy_\alpha$ an \emph{orthotube}, and its length is defined to be the distance between the intersection points with $\Yy_\gamma$ and $\Yy_\delta$. We denote by $\ort_\Sigma$ the collection of orthotubes and by $\ort_\Sigma(\gamma)$ the subset of orthotubes orthogonal to a fixed $\Yy_\gamma$. If $\alpha\in \ort_\Sigma(\gamma)$, we denote by $\delta_\alpha$ the peripheral element associated to $\alpha$ and different from $\gamma$.

The length of orthotubes has also geometric significance: in Section \ref{sec:doubles} we  construct what we call the holomorphic double of a maximal representation (with suitable hypotheses on the peripheral elements). The double is a specific representation of the fundamental group of the double of the surface that restrict to the given representation. Then we can show that the Finsler length of an orthotube is half the length of the corresponding curve in the double (Proposition \ref{prop:4.6}).

Given $v\in \aw$, we denote by $v_n$  the smallest coordinate of $v$. Our geometric generalizations of Basmajian's identity to the setting of maximal representations are the following:
\begin{theorintro}\label{thm:main} 
For any maximal representation $\rho:\G\to\Sp(2n,\R)$ with the property that the image of peripheral elements are Shilov hyperbolic, we have 
\[
2n\sum_{\alpha\in\ort_\Sigma}\log \coth \frac{\ellC(\alpha)_n}{2}\geq\ellF(\partial\Sigma)\geq 2n \sum_{\alpha\in\ort_\Sigma}\log\coth\frac{\ellF(\alpha)}{n} \label{Finsler} \tag{A1}
\]
and
\[
4\sqrt{n}\sum_{\alpha\in\ort_\Sigma}\log \coth \frac{\ellC(\alpha)_n}{2}\geq \ellR(\partial\Sigma)\geq 4\sqrt{n}\sum_{\alpha\in\ort_\Sigma}\log \coth \frac{\ellR(\alpha)}{2\sqrt{n}} \label{Riemannian} \tag{A2}\]
with equalities if and only if $\rho$ is, up to a character in a compact group, the composition of a holonomy representation of a hyperbolization into $\SL(2,\R)$ with the diagonal representation of $\SL(2,\R)$ into $\Sp(2n,\R)$.
\end{theorintro}
Note that adding the condition on the images of peripheral elements is equivalent to requiring that the representation is Anosov (see Section \ref{maxreps}). Interestingly, for both metrics the difference between the middle term and the right term can be arbitrarily large (Proposition \ref{prop:arbitrarilybad}), namely there are sequences of maximal representations in which the length of the boundary components stay bounded away from zero, but the  $\R$-tubes associated to any two peripheral elements are arbitrarily far apart. This is discussed in  Section \ref{arbitrarilybad}.

\begin{Remark} If we drop the hypothesis on the Shilov hyperbolic image, the right-hand side inequalities in \eqref{Finsler} and \eqref{Riemannian} still hold, as long as there is at least one peripheral element whose image is Shilov hyperbolic. Moreover, the characterization of any equality still holds under these milder assumptions.
\end{Remark}

\begin{Remark}
A natural question is how this work generalizes to maximal representations $\rho$ in other Hermitian Lie groups $G$. It is well known that then the Zariski closure $H=\ov{\rho(\G)}^Z$ of the image of $\rho$ is a Hermitian Lie group of tube type \cite[Theorem 5]{biw}. Under the mild assumption that $H$ has no factor isogenous to $\SL(2,\R)$ or ${\rm E}_{6(-14)}$, there exists a (virtual) tight holomorphic embedding $$\iota:H\to \Sp(2n,\R)$$ (for some explicit $n$ depending on $H$ only \cite[Example 8.7]{BIWtight}) that is isometric, up to suitably rescaling the metric on the irreducible factors of $H$. Furthermore the Shilov boundary of $H$ is naturally a subspace of $\Ll(\R^{2n})$ by \cite[Theorem 7]{BIWtight}.

The composition $\rho'=\iota\circ\rho:\G\to \Sp(2n,\R)$ is then a maximal representation and Theorem \ref{thm:main} holds for $\rho'$. Note that, if $H$ is irreducible, the Riemannian translation distance for $\rho'$ is an explicit multiple of the Riemannian translation distance for the action $\rho$ of $\Gamma$ on the symmetric space $\Xx_H$ associated to $H$.  Since there is a $\iota$-equivariant totally geodesic holomorphic inclusion  $\Xx_H\to\Xx$, it is possible to verify that, for every orthotube $\Yy_\alpha$ that we consider, the endpoints of $\Yy_\alpha$ belong to the Shilov boundary of $H$. This follows from Lemma \ref{existenceorth} and the observation that the totally geodesic holomorphic and tight image of a polydisk in $\Xx_H$ is a partially diagonal subset of a polydisk of $\Xx$ (the rank of $\Xx_H$ is in general smaller than the rank of $\Xx$). In turn this implies that intersection $\Yy_\alpha\cap \Xx_H$ is non-empty, consists of the parallel set of a maximally singular 
geodesic in $\Xx_H$ (an $H$-tube) and is orthogonal to the $H$-tubes in $\Xx_H$ associated to the two peripheral elements corresponding to $\alpha$. 

It is furthermore possible to verify that this does not generalize to the Fuchsian locus within Hitchin maximal representations: this is not a contradiction since the totally geodesic equivariant map is, in that case, not holomorphic. This also partially justifies why, for representations in such Fuchsian locus, equality in Theorem \ref{thm:main} does not hold.
\end{Remark}

As a corollary of our main result  we are able to deduce an interesting geometric property of the locally symmetric space $\rho(\Gamma)\backslash \Xx$: not only the image $\langle\rho(\gamma)\rangle\backslash\Yy_\gamma$ in $\rho(\Gamma)\backslash \Xx$ of $\Yy_\gamma$ is an embedded manifold, but it also admits an embedded tubular neighborhood whose width can be explicitly computed as a function of the translation length of $\Gamma$:   
\begin{cor}\label{corB}
If $\g\in\G$ corresponds to a simple closed curve or boundary component of $\Sigma$ and $\rho:\Gamma\to \Sp(2n,\R)$ is an Anosov maximal representation, then $\langle\rho(\gamma)\rangle\backslash\Yy_\gamma\subset\rho(\Gamma)\backslash \Xx$ has an embedded tubular neighbourhood $C(\g)$ of width
$$w(\g):=\sqrt{n}\arccoth\left(\exp\left(\frac{\ellR(\g)}{2\sqrt{n}}\right)\right)$$
with respect to the Riemannian metric. Moreover, if $\delta$ corresponds to a simple closed curve or boundary component that is disjoint from $\g$, then the neighborhoods $C(\g)$ and $C(\delta)$ are disjoint.
\end{cor}

The main difference of our approach with respect to Vlamis and Yarmola's one is that with our inequalities we want to relate intrinsic geometric quantities of the locally symmetric space associated to a representation. Instead, following Labourie--McShane, they mainly work with algebraic versions of the identities and use cross-ratios to translate Basmajian's result in the language of representations. As a partial step towards the proof of Theorem  \ref{thm:main} we prove the analogue of Vlamis and Yarmola's result in our context:
\begin{theorintro}\label{thm:cr}
Let $\rho:\Gamma\to\Sp(2n,\R)$ be a maximal representation with the property that the image of each peripheral element is Shilov hyperbolic. Then for every peripheral element $\gamma\in\Gamma$ we have
$$\ell_{\B}(\gamma)=\sum_{\alpha\in\ort_\Sigma(\gamma)}\log\B(\gamma^-,\delta_\alpha^+,\gamma^+,\delta_\alpha^-),$$
where $\B$ is the $\R$-valued cross-ratio constructed by Labourie in \cite{Lab_energy}.
\end{theorintro}
Our proof of this result is very different from Vlamis and Yarmola's proof: they build on the fact that the image of the boundary map associated to a Hitchin representation is $C^1$, which is in general far from being true in the case of maximal representations. Instead, we adapt a more geometric proof which could also be used to get estimates on the Hausdorff dimension of the image of the boundary map.  

We conclude this introduction by mentioning that Xu \cite{xu_pressure} used orthogeodesics to study the metric completion of the pressure metric (see \cite{pressure}) on the Teichm\"uller space of surfaces with boundary. He proved that, in that case, the pressure metric is not a constant multiple of the Weil-Petersson metric. We hope that our study will have application in the study of the pressure metric on the space of maximal representations as well.  
\subsection{Plan of the paper}
In Section \ref{sec:symspace} we discuss properties of the geometry of the symmetric space associated to $\Sp(2n,\R)$ and of the synthetic geometry of $\R$-tubes. We recall the results of \cite{bp} and expand them when needed. In particular we relate, in Section \ref{sec:producttube} the  Finsler length and the translation on the Euclidean factor of an $\R$-tube (Lemma \ref{lem:producttube}), we define causal paths in $\R$-tubes (these will arise naturally while dealing with maximal representations), and give an explicit bound on the length of the projection of a causal path on the symmetric space for $\SL(n,\R)$ in terms of its length in the Euclidean factor  (Lemma \ref{lem:pos}).
Section \ref{maxreps} contains the necessary preliminaries about maximal representations and the construction of the holomorphic double of a representation (Proposition \ref{prop:double}).
In Section \ref{sec:ort} we define our generalization of orthogeodesic in the context of maximal representations and prove the relation between the length of an orthotube and the length of the associated element in the double of the representation (Proposition \ref{prop:4.6}).
Section \ref{sec:strategy} describes the idea of the proof of Basmajian's identity and the difficulties that arise when generalizing it to higher rank.
Section \ref{sec:0measure} is devoted to the proof of Theorem \ref{thm:cr}, generalizing the classical strategy described in Section \ref{classicalmeasure0}.
In the final section we prove the geometric inequalities announced in the introduction and the characterization of diagonal representations (Theorem \ref{thm:main}).
\subsection*{Acknowledgements}
The second author would like to thank Ursula Hamenst\"adt for asking if the equality in the Collar lemma of \cite{bp} characterizes the Fuchsian locus -- that question gave the initial motivation for this work. We are grateful to Brian Bowditch for useful conversations and in particular for suggesting a proof of Proposition \ref{classicalmeasure0} that turned out to be amenable to generalization to maximal representations, to the referees for many detailed suggestions, and to Jean-Louis Clerc for asking about possible generalizations to other symmetric domains.\\
We acknowledge support of Swiss National Science Foundation grants number P2FRP2\_161723 (Federica Fanoni) and P2EZP2\_159117 (Maria Beatrice Pozzetti).
\section{The symmetric space associated to $\Sp(2n,\R)$}\label{sec:symspace}
Recall that the symplectic group $\Sp(2n,\R)$ is the subgroup of $\SL(2n,\R)$ preserving the symplectic form $\omega(\cdot,\cdot)$ represented, with respect to the standard basis, by the matrix $$J_n=\bsm 0&\Id_n\\-\Id_n&0\esm.$$

The symmetric space $\calX$ associated to the symplectic group $\Sp(2n,\R)$ is often referred to as \emph{Siegel space}. In this paper we will be concerned with locally symmetric spaces arising as the quotient of $\calX$ by the image $\rho(\G)<\Sp(2n,\R)$ of a maximal representation. We will be interested in two models for $\calX$: the \emph{upper-half space} and the image of the \emph{Borel embedding}.

The upper-half space is the generalization of the upper-half plane model of the hyperbolic plane, given by a specific set of symmetric matrices:
$$\calX=\{X+iY|X\in\Sym(n,\R), Y\in \Sym^+(n,\R)\},$$
where $\Sym(n,\R)$ denotes the set of $n$-dimensional symmetric matrices with coefficients in $\R$ and $\Sym^+(n,\R)$ is the subset of $\Sym(n,\R)$ given by positive definite matrices. The group $\Sp(2n,\R)$ acts by fractional linear transformations in this model:
$$\left(\begin{matrix}
A&B\\C&D
\end{matrix}\right)\cdot Z=(AZ+B)(CZ+D)^{-1}.$$
The image of the Borel embedding is
$$\mathbb X=\{l\in\Ll(\C^{2n})\,|\,\left.i\omega(\cdot,\sigma(\cdot))_\C\right|_{l\times l}\; \mbox{is positive definite}\}.$$
Here $\Ll(\C^{2n})$ is the set of Lagrangians, the maximal isotropic subspaces of $\C^{2n}$ for the complexification of  the symplectic form $\omega(\cdot,\cdot)_\C$, and $\sigma:\C^{2n}\to\C^{2n}$ denotes the complex conjugation.

We will consider the affine chart $\iota:\Sym(n,\C)\to\Ll(\C^{2n})$ that associates to a symmetric matrix $Z$ the linear subspace of $\C^{2n}$ spanned by the columns of the matrix $\bsm Z\\\Id_n\esm$. It is easy to verify that $\iota$ is well defined and induces an $\Sp(2n,\R)$-equivariant identification  $\iota:\calX\to \mathbb X$ (cfr. \cite[Section 2.2]{bp} for more detail). The complex conjugation $\sigma:\C^{2n}\to\C^{2n}$ induces a map on $\Ll(\C^{2n})$ that will still be  denoted by $\sigma$ with a slight abuse of notation. It is easy to verify that $\sigma$ corresponds, via $\iota$, to the complex conjugation on $\Sym(n,\C)$.

A \emph{maximal polydisk} in $\calX$ is the image of a totally geodesic and holomorphic embedding of the Cartesian product of $n$ copies of the Poincar\'e disk into $\calX$. Maximal polydisks exist and they are all conjugate under the action of $\Sp(2n,\R)$ (see \cite[p.\ 280]{wolf}). Polydisks arise as complexifications of maximal flats. Thus each pair of points $(x,y)$ is contained in a maximal polydisk, that is unique if the direction determined by $(x,y)$ is regular.
\subsection{Lagrangians and boundaries}

The set of real Lagrangians $\Ll(\R^{2n})$ naturally arises as the unique closed $\Sp(2n,\R)$ orbit in the boundary of $\calX$ in its Borel embedding and for this reason $\Ll(\R^{2n})$ is the Shilov boundary of the bounded domain realization of $\Xx$ (see \cite{wienhard_thesis} for details). The restriction of the affine chart $\iota$ to the subspace $\Sym(n,\R)$ provides a parametrization of the set of real Lagrangians that are transverse (as linear subspaces) to $\<e_1,\ldots,e_n\>$. We will denote $\<e_1,\ldots,e_n\>$ by $l_\infty$, since it is at infinity in the affine chart we chose. Whenever this won't generate confusion we will omit the map $\iota$ and identify symmetric matrices (with real or complex coefficients) with Lagrangian subspaces (of $\R^{2n}$ and $\C^{2n}$ respectively).

Since $\calX$ has higher rank for $n>1$, the visual boundary $\partial_\infty\calX$ is not homogeneous \cite{Eberlein}, and the relation between the visual boundary and the closure of $\calX$ in the Borel embedding is in general fairly complicated. However there is a closed orbit in $\partial_\infty\calX$ which is naturally identified with the Lagrangians \cite[Theorem 9.11]{Loos}. In particular this allows us to associate to a point $l\in\Ll(\R^{2n})$ a class of asymptotic directions.

Denote by $\Ll(\R^{2n})^{(k)}$ the set of $k$-tuples of pairwise transverse Lagrangians. It is well known and easy to prove that the group $\Sp(2n,\R)$ acts transitively $\Ll(\R^{2n})^{(2)}$. Moreover, it has $(n+1)$ orbits in $\Ll(\R^{2n})^{(3)}$, indexed by the Maslov cocycle \cite[Section 1.5]{LV}. The value of the Maslov cocycle is maximal on the orbit of  
$$(\langle e_1,\ldots,e_n\rangle,\langle e_{n+1},\ldots,e_{2n}\rangle,\langle e_1+e_{n+1},\ldots, e_n+e_{2n}\rangle)=(l_\infty, 0,\Id).$$
\begin{definition}
A triple of pairwise transverse Lagrangians is called \emph{maximal} if it is in the $\Sp(2n,\R)$-orbit of $(l_\infty, 0,\Id)$.
\end{definition}
Maximal triples should be regarded as a generalization of positively oriented triples in the circle. As we will shortly see, maximal triples play a fundamental role in the definition and study of maximal representations. We will need a concrete criterion to check when triples of Lagrangians are maximal. The following is well known, and a proof can be found, for example, in \cite[Lemma 2.10]{bp}.
\begin{lem}\label{lem:max}The following hold:
\begin{enumerate}
\item any cyclic permutation of a maximal triple is maximal;
\item the triple $(l_\infty,X,Y)$ is maximal if and only if $Y-X$ is positive definite;
\item if $Z-X$ is positive definite, the triple $(X,Y,Z)$ is maximal if and only if $Z-Y$ and $Y-X$ is positive definite.
\end{enumerate}
\end{lem}

More generally, we can define maximal $m$-tuples:
\begin{definition} An $m$-tuple $(l_1,\dots,l_m)\in\Ll(\R^{2n})^{(m)}$ is \emph{maximal} if for every $i<j<k$ the triple $(l_i,l_j,l_k)$ is maximal.
\end{definition}
We will often consider maximal $4$-tuples and it will be useful to consider special representatives in an $\Sp(2n,\R)$-orbit of a maximal $4$-tuple. Two representatives are described in \cite[Prop.\ 2.11]{bp} and another is given in the following Lemma:
\begin{lem}\label{lem:IdLambda}
If $(l_1,l_2,l_3,l_4)$ is a maximal $4$-tuple, there exists $\Lambda=\diag(\l_1,\dots\l_n)$ with $\l_i\in(0,1)$ and an element $g\in\Sp(2n,\R)$ such that
$$g(l_1,l_2,l_3,l_4)=(-\Id,-\Lambda,\Lambda,\Id).$$
\end{lem}
\begin{proof}
By \cite[Prop.\ 2.11]{bp}, we can assume $(l_1,l_2,l_3,l_4)=(-\Id,0,D,l_\infty)$ for a diagonal matrix $D=\diag(d_1,\dots,d_n)$ with $d_1\geq \dots\geq d_n>0$. Note that for every $d_i>0$, there exists $\l_i=\l_i(d_i)$ and a matrix $A_i=A_i(d_i)=\left(\begin{array}{cc}
\alpha_i & \beta_i\\
\gamma_i & \delta_i
\end{array}\right)$ such that $A_i$ sends the $4$-tuple $(-1,0,d_i,\infty)\in (\partial\hyp)^4$ to $(-1,-\l_i,\l_i,1)$.
Consider the element $g\in\Sp(2n,\R)$ given by
$$g=\left(\begin{array}{cc}
\diag(\alpha_1,\dots,\alpha_n)& \diag(\beta_1,\dots,\beta_n)\\
\diag(\gamma_1,\dots,\gamma_n)& \diag(\delta_1,\dots, \delta_n)
\end{array} \right).$$
Then it is straightforward to check that $$g(-\Id,0,D,l_\infty)=(-\Id,-\Lambda,\Lambda,\Id),$$ where $\Lambda=\diag(\lambda_1,\dots,\lambda_n)$.
\end{proof}
\subsection{$\Sp(2n,\R)$-invariant distances}\label{metrics}
We are interested in three $\Sp(2n,\R)$-invariant distances on the symmetric space $\calX$: the \emph{vectorial distance}, the \emph{Riemannian distance} and the \emph{(determinant) Finsler distance}.

Fix a point $p$ in a maximal flat $F$ and a Weyl chamber $\aw\subset T_pF$. This is a fundamental domain for the action of $\Sp(2n,\R)$ on the tangent bundle $T\calX$. In our case we have
$$\aw=\{(x_1,\ldots,x_n)\in\R^n|x_1\geq\ldots\geq x_n\geq 0\}.$$
A vector in the model Weyl chamber is \emph{regular} if all the inequalities are strict, which is equivalent to being contained in a unique flat. We say that  $v$ is \emph{singular} pointing in the direction of a Lagrangian if $x_1=\ldots=x_n>0$. Indeed, if a vector $v$ is singular pointing in the direction of a Lagrangian and $\g$ is the geodesic  determined by exponentiating  $v$, the endpoints of $\g$ in the visual boundary $\partial_\infty\calX$ are two Lagrangians.

In order to define the projection onto the model Weyl chamber, we need to recall from \cite{bp} the definition of an endomorphism-valued cross-ratio.
If two real or complex Lagrangians $l_1$ and $l_2$ are transverse (denoted by $l_1\pitchfork l_2$), we denote by $p_{l_1}^{\parallel l_2}:\R^{2n}\to l_1$ (resp. by $p_{l_1}^{\parallel l_2}:\C^{2n}\to l_1$) the projection to $l_1$ parallel to $l_2$. 
\begin{definition}For Lagrangians $l_1,\dots, l_4\in\Ll(\C^{2n})$ such that $l_1\pitchfork l_2 $ and $l_3\pitchfork l_4$, the \emph{cross-ratio} $R(l_1,l_2,l_3,l_4)$ is the endomorphism of $l_1$ given by $$R(l_1,l_2,l_3,l_4)=\left. p_{l_1}^{\parallel l_2}\circ p_{l_4}^{\parallel l_3}\right|_{l_1}.$$ 
\end{definition}
\noindent We will use multiple times the explicit expression for the cross-ratio $R$ on the affine chart $\iota(\Sym(n,\C))$ of $\Ll(\C^{2n})$ (cfr.\ \cite[Lemma 4.2]{bp}):
\begin{equation} \label{cross-ratio}
R(X_1,X_2,X_3,X_4)=(X_1-X_2)^{-1}(X_4-X_2)(X_4-X_3)^{-1}(X_1-X_3).
\end{equation}
Here $R$ is expressed with respect to the basis of $X_1$ given by the columns of the matrix $\bsm X_1\\\Id_n\esm$.

Since the symmetric space $\Xx$ is a complete, negatively curved Riemannian manifold, any pair of points $(a,b)\in \calX^2$ is joined by a unique geodesic $\gamma$ and therefore corresponds to a unique vector $v\in T_a\Xx$. In \cite{Siegel}, Siegel proved that the cross-ratio we just introduced can be used to describe the projection of a pair of points in $\Xx$ onto the Weyl chamber, the fundamental domain for the $\Sp(2n,\R)$-action on $T\calX$:
\begin{theor}[\cite{Siegel}]
The projection onto the Weyl chamber is given by
\begin{align*}
\calX^2&\to \aw\\
(X,Z)&\mapsto (\log(\l_1),\dots,\log(\l_n))
\end{align*}
where $\l_i =\frac{1+\sqrt r_i}{1-\sqrt r_i}$ and  $1>r_1\geq \ldots\geq r_n\geq0$ are the eigenvalues of $R(X,\ov Z,Z,\ov X)$.\end{theor}
We are now ready to define the distances on $\calX$.
\begin{definition}The \emph{vectorial distance} $\dC$ is the projection onto the Weyl chamber $\aw$.
\end{definition}
The vectorial distance is not $\R$-valued; however there is a natural partial order on the Weyl chamber that allows us to talk about triangular inequality. This, together with many interesting properties of $\dC$, is proven by Parreau in \cite{parreau_distvect}. Since in the symmetric space associated to $\Sp(2n,\R)$ the opposition involution is trivial, the vectorial distance is symmetric: for each pair $x,y\in\calX$ it holds $\dC(x,y)=\dC(y,x)$.

As the Weyl chamber is a fundamental domain for the $\Sp(2n,\R)$-action on $\calX^2$, any $\Sp(2n,\R)$-invariant distance can be obtained composing the vectorial distance with a suitable function. We will use two $\R$-valued distances:
\begin{definition}The \emph{Riemannian distance} $d^R$ is the composition of $\dC$ with the function
\begin{align*}
\aw&\to\R\\
(x_1,\ldots,x_n)&\mapsto \sqrt{\sum x_i^2}.
\end{align*}
\end{definition}
Notice that the Riemannian distance is the distance induced by the unique $\Sp(2n,\R)$-invariant metric on $\calX$ with minimal holomorphic sectional curvature equal to $-1$. The normalization is chosen so that the polydisks are isometrically and holomorphically embedded. 
\begin{definition}
The \emph{(determinant) Finsler metric} $d^F$ is the composition of $\dC$ with the function
\begin{align*}
\aw&\to\R\\
(x_1,\ldots,x_n)&\mapsto\frac 12 \sum x_i.
\end{align*}
\end{definition}
An easy application of Cauchy-Schwarz's inequality shows:
\begin{lem}\label{lem:relRF}
For every $a,b\in\calX$
 $$\dR(a,b)\leq 2\dF(a,b)\leq \sqrt n\dR(a,b).$$
\end{lem}

As we will see, the computation of Finsler and vectorial distance often reduces to computations about eigenvalues of positive definite symmetric matrices. We will denote by $\langle \cdot,\cdot\rangle$ the standard Euclidean inner product on $\R^{2n}$. One of the tools we will use many times is the (Courant--Fischer--Weyl) min--max principle, which we recall here:
\begin{lem}\label{lem:minmax}
Let $A$ be a symmetric matrix of eigenvalues $a_1\geq\dots\geq a_n$. Then
$$a_i=\min_{\dim(U)=n-i+1}\max\left\{\frac{\langle Av,v\rangle}{\|v\|^2}\, |\, 0\neq v\in U\right\}=\max_{\dim(U)=i}\min\left\{\frac{\langle Av,v\rangle}{\|v\|^2}\, |\, 0\neq v\in U\right\}.$$
In particular
$$a_1=\max_{v\neq 0}\frac{\langle Av,v\rangle}{\|v\|^2}\;\;\; \mbox{and}\;\;\; a_n=\min_{v\neq 0}\frac{\langle Av,v\rangle}{\|v\|^2}.$$
\end{lem}
The following lemmas will be handy:
\begin{lem}\label{lem:evAB}
Let $A,B$ be positive definite symmetric matrices. Then
$$\max\ev A\max\ev B\geq \max \ev (AB)\geq \max \ev A\min \ev B.$$
\end{lem}
\begin{proof}
Notice that since $B$ is symmetric and positive definite it admits a positive square root. Moreover, since the eigenvalues of a matrix are invariant under conjugation, the eigenvalues of $AB$ coincide with the eigenvalues of $B^{1/2}A B^{1/2}$. 
The result then follows using the min--max principle:

 \begin{align*}
  \max\ev(B^{1/2}A B^{1/2})&=\max_{v\in\R^n\setminus 0}\frac{\langle B^{1/2}A B^{1/2}v,v\rangle}{\|v\|^2}\\&=\max_{v\in\R^n\setminus 0}\frac{\langle B^{1/2}A B^{1/2}v,v\rangle}{\|B^{1/2}v\|^2}\frac{\langle B^{1/2}v, B^{1/2}v\rangle}{\|v\|^2}\\
  &\geq \max_{w\neq 0}\frac{\rangle Aw,w\langle}{\|w\|^2}\min_{v\neq 0}\frac{\rangle Bv,v\langle}{\|v\|^2}\\
  &=\max\ev(A)\min\ev(B).
 \end{align*}
\end{proof}

\begin{lem}\label{lem:eigenvalues}
Let $A,B$ be positive definite symmetric matrices. Then the difference $A-B$ is positive definite if and only if all eigenvalues of $AB^{-1}$ are bigger than one. 
\end{lem}
\begin{proof}
This follows from the observation that
$$B^{-1/2}A B^{-1/2}=B^{-1/2}(A-B) B^{-1/2}+\Id$$
and the fact that a symmetric matrix  $M$ is positive definite if and only if $^tNMN$ is positive definite for any invertible matrix $N$.
\end{proof}
Using Lemma \ref{lem:eigenvalues}, we get:
\begin{lem}\label{Finslerdistance}
Let $A$ and $B$ be positive definite symmetric matrices such that the difference $B-A$ is positive definite. Let $\mu_1\geq\dots\geq \mu_n$ be the eigenvalues of $A^{-1}B$. Then
$$\dC(iA,iB)=(\log \mu_1,\dots,\log \mu_n)$$
and
$$\dF(iA,iB)=\frac{1}{2}\log\det(A^{-1}B).$$
\end{lem}

\begin{proof}
Note first that by Lemma \ref{lem:eigenvalues} the eigenvalues of $A^{-1}B$ are all bigger than one. We then have
$$\dC(iA,iB)=(\log\lambda_1,\ldots,\log\lambda_n)$$
where $\lambda_i=\frac{1+\sqrt{r_i}}{1-\sqrt{r_i}}$ and the $r_i$ are the eigenvalues of $R(iA,-iB,iB,-iA)$. Using \eqref{cross-ratio}, one can compute that
$$R(iA,-iB,iB,-iA)=(\Id+A^{-1}B)^{-2}(A^{-1}B-\Id)^2,$$
so $r_i=\frac{(\mu_i-1)^2}{(\mu_i+1)^2}$. But since all the $\mu_i$ are all bigger than one, we deduce that $\lambda_i=\mu_i$. Hence
$$\dC(iA,iB)=(\log \mu_1,\dots,\log \mu_n)$$
and
$$\dF(iA,iB)=\frac{1}{2}\sum_{j=1}^n \log \mu_j=\frac{1}{2}\log\prod_{j=1}^n \mu_j=\frac{1}{2}\log \det (A^{-1}B).$$
\end{proof}
\subsection{$\R$-tubes}\label{Rtubes}
In this section we recall the definition of $\R$-tubes, subspaces of $\calX$ that play the role of geodesics in $\hyp$, and we will prove some results that we will use in the remainder of the paper.
Let $\{a,b\}$ be an unordered pair of transverse real Lagrangians. 
\begin{definition} The \emph{$\R$-tube} $\Yy_{a,b}$ associated to $\{a,b\}$ is the set
$$\Yy_{a,b}=\{l\in\calX \;|\; R(a,l,\sigma ({l}),b)=-\Id\}.$$
We will refer to the real Lagrangians $a,b$ as  the \emph{endpoints} of $\Yy_{a,b}$.
\end{definition}
It can be proven (see \cite[Section 4.2]{bp}) that $\Yy_{a,b}$ is a totally geodesic subspace of $\calX$ of the same real rank as $\calX$ and that it is the parallel set of the Riemannian singular geodesics, whose endpoints in the visual boundary of $\calX$ are the Lagrangians $a$ and $b$.  The stabilizer of $\Yy_{a,b}$ is $\Stab_{\Sp(2n,\R)}(\{a,b\})$ and is a $\mathbb Z/2\mathbb Z$-extension of $\GL(n,\R)$ that acts transitively on the $\R$-tube.

Up to the symplectic group action we can reduce to a model $\R$-tube, the one with endpoints $0
$ and $l_\infty
$. In the upper-half space model this standard tube consists of matrices of the form $$\Yy_{0,l_\infty}=\{iY|\; Y\in \Sym^+(n,\R)\}$$
and $\GL_n(\R)$ acts on it as $G\cdot i X=i(GX(^t\!G))$.

It was exploited in \cite{bp} (cfr.\ also \cite{bilw}) that the incidence structure of  $\R$-tubes in the Siegel space forms a synthetic geometry that shares many common features with the hyperbolic geometry in $\mathbb H^2$. In particular, the following result shows how the intersection pattern of $\R$-tubes reflects the intersection pattern of geodesics in the hyperbolic plane.
\begin{prop}\label{intersectingtubes}
If $(l_1,l_2,l_3,l_4)$ is maximal, the intersection $\Yy_{l_1,l_3}\cap\Yy_{l_2,l_4}$ consists of a single point and $\Yy_{l_1,l_2}\cap\Yy_{l_3,l_4}$ is empty.
\end{prop}
\begin{proof}
The first result is proven in \cite[Lemma 4.7]{bp}. With the same techiques we can prove that the second intersection is empty. We reproduce the argument for completeness.

Up to the action of the symplectic group, we can assume that $$(l_1,l_2,l_3,l_4)=(-\Id,\Lambda, 0,l_{\infty}),$$
where $\Lambda=\diag(\l_1,\dots,\l_n)$ and $\l_i\in(-1,0)$ (cfr.\ \cite[Proposition 2.11]{bp}). Now a point $y$ belongs to $\Yy_{l_1,l_2}\cap\Yy_{l_3,l_4}=\Yy_{-\Id,\Lambda}\cap\Yy_{0,\ell_\infty}$ if and only if $y=iY$, for some $Y\in\Sym(n,\R)$, and $iY\in \Yy_{-\Id,\L}$. By definition $iY\in \Yy_{-\Id,\L}$ if and only if $R(-\Id,iY,-iY,\L)=-\Id$. Using \eqref{cross-ratio} we have
\begin{align*}
R(-\Id,iY,-iY,\L)=-\Id &\Leftrightarrow (-\Id-iY)^{-1}(\L-iY)(\L+iY)^{-1}(-\Id+iY)=-\Id\\ 
&\Leftrightarrow (\L-iY)(\L+iY)^{-1}=(\Id+iY)(-Id+iY)^{-1}\\
&\Leftrightarrow (\L-iY)(\L+iY)^{-1}=(-\Id+iY)^{-1}(\Id+iY)\\
&\Leftrightarrow(-\Id+iY)(\L-iY)=(\Id+iY)(\L+iY)\\
&\Leftrightarrow Y^2=\L.
\end{align*}
But as $\Lambda$ is negative definite, there is no solution to this equation.
\end{proof}
\subsection{The product structure of a tube and causal maps}\label{sec:producttube}
We now turn to a more precise description of the geometry of a single $\R$-tube. Recall that the standard model for the symmetric space associated to $\GL(n,\R)$ is
$$\Xx_{\GL(n,\R)}=\Sym^+(n,\R)$$
and $\GL(n,\R)$ acts on $\Xx_{\GL_n(\R)}$ by $G\cdot X=GX^tG$.
We endow $\Xx_{\GL(n,\R)}$ with the Riemannian distance given by
$$d_\GL(X,Y)=\sqrt{\sum_{i=1}^n (\log\l_i)^2},$$
where $\l_i$ are the eigenvalues of $XY^{-1}$. This is twice the normalization for the Riemannian distance chosen in \cite{benoist_proprietes, parreau_compactification}, but it is better suited to our purposes because with this choice the natural identification $\Xx_{\GL(n,\R)}\cong \Yy_{0,l_\infty}$ is an isometry (where $\Yy_{0,l_\infty}$ is equipped with the Riemannian metric).

Recall that a model for the symmetric space associated to $\SL(n,\R)$ is
$$\Xx_{\SL(n,\R)}=\{X\in\Sym^+(n,\R)|\; \det(X)=1\}$$
and its model Weyl chamber is
$$\aw_{\SL(n,\R)}=\left\{(x_1,\ldots,x_n)\in\R^n|x_1\geq\ldots\geq x_n,\sum_{i=1}^n x_i=0\right\}.$$
We associate to the pair $(X,Y)$ the vector $(\log \l_1,\ldots,\log\l_n)$ where $\l_1\geq\ldots\geq\l_n$ are the eigenvalues of $XY^{-1}$. We normalize the Riemannian metric on $\Xx_{\SL(n,\R)}$ so that
$$\dSL(X,Y)=\sqrt{\sum_{i=1}^n(\log \l_i)^2}.$$ Coherently with above, this is twice the standard normalization. 

The group $\GL(n,\R)$ is reductive and its symmetric space, endowed with the Riemannian metric, splits as the direct product $\Xx_{\GL(n,\R)}=\R\times\Xx_{\SL(n,\R)}$; $\R$ is the Euclidean factor of the reducible symmetric space $\Xx_{\GL(n,\R)}$. 
Explicitly:
\begin{lem}\label{lem:producttube}
The map
\begin{align*}
F=\pi^\R\times\pi^{\rm SL}:\Xx_{\GL(n,\R)}&\longrightarrow \R\times \Xx_{\SL(n,\R)}\\
X&\longmapsto\left(\frac {\log \det X}{\sqrt{n}},\left(\frac1{(\det X)^{1/n}}X\right)\right)
\end{align*}
is an isometry.
\end{lem}
\begin{proof}
Clearly the map is a bijection, so we just need to show that it preserves distances. For any $X,Y\in\Xx$ we have
$$|\pi^\R(X)-\pi^\R(Y)|=\left|\frac{\log\det X-\log\det Y}{\sqrt{n}}\right|=\left|\frac{\log\det(XY^{-1})}{\sqrt{n}}\right|.$$
Moreover
$$\dSL(\pi^\SL(X),\pi^\SL(Y))=\sqrt{\sum_{i=1}^n(\log\mu_i)^2}$$
where $\mu_1\geq\dots\geq \mu_n $ are the eigenvalues of
$$\left(\frac1{(\det X)^{1/n}}X\right)\left(\frac1{(\det Y)^{1/n}}Y\right)^{-1}=\frac1{(\det (XY^{-1}))^{1/n}}XY^{-1}.$$
So
$$\mu_i=\frac1{(\det (XY^{-1}))^{1/n}}\l_i$$
where the $\l_i$ are the eigenvalues of $XY^{-1}$. Set $d:=\det(XY^{-1})$. We have:
\begin{align*}
\mbox{d}(F(X),F(Y))^2&=|\pi^\R(X)-\pi^\R(Y)|^2+\dSL(\pi^\SL(X),\pi^\SL(Y))^2\\
&=\frac{1}{n}(\log d)^2+\sum_{i=1}^n\left(
\log \l_i-\frac{1}{n}\log d\right)^2\\
&=\sum_{i=1}^n(\log \l_i)^2+\frac{2}{n}\log d \underbrace{(\log d-\sum_{i=1}^n\log \l_i)}_{=0}=\dGL(X,Y)^2
\end{align*}
\end{proof}

\begin{remark}
On the model flat $\pi^\R$ is the scalar product with the unit vector  $(1/\sqrt n,\ldots,1/\sqrt n)$ and the tangent space to $\SL(n,\R)$ is its orthogonal.
\end{remark}

%
We now define causal maps and show some of the properties we will need.
\begin{definition}\label{def:Causal}Let $K$ be a subset of $\R$; a map $f:K\to \Xx_{\GL(n,\R)}$ is \emph{causal} if for each pair $x<y\in K$, $f(y)-f(x)$ is positive definite.
\end{definition}
The next lemma summarizes a property of causal paths that will be crucial for our future analysis: for any causal path the length of its projection on the Euclidean factor gives an upper bound on the length of the projection to $\Xx_{\SL(n,\R)}$:
\begin{lem}\label{lem:pos}
Let $A,B\in\Xx_{\GL(n,\R)}$. If $B-A$ is positive definite, then
$$\pi^\R(B)>\pi^\R(A)$$ and $$\sqrt {n-1}(\pi^\R(B)-\pi^\R(A))>\dSL( \pi^{\SL}(B),\pi^{\SL}(A)).$$
\end{lem}
\begin{proof}
It follows from Lemma \ref{lem:eigenvalues} that if $B-A$ is positive definite then all eigenvalues of $BA^{-1}$ are bigger than 1. In particular $\det B>\det A$, which gives the first inequality.

In order to verify the second statement we can assume that $A=\Id$: indeed $\GL(n,\R)$ acts transitively on $\Xx_{\GL(n,\R)}$ by isometries, and since the splitting of Lemma \ref{lem:producttube} is isometric, both terms in the second inequality are invariant by the $\GL(n,\R)$-action. Denoting by $b_1\geq\ldots\geq b_n>1$ the eigenvalues of $B$ we have $\pi^\R(B)=\frac{1}{\sqrt{n}}\sum_{i=1}^n \log b_i$, $\pi^\R(A)=0$, and the vector in $\aw_{\SL(n,\R)}$ associated to the pair $( \pi^{\SL}(B),\pi^{\SL}(A))$ is $$\left(\log b_1-\frac {\pi^\R(B)}{\sqrt{n}}, \ldots, \log b_n-\frac{\pi^\R(B)}{\sqrt{n}}\right).$$
Then
\begin{align*}
\dSL( \pi^{\SL}(B),\pi^{\SL}(A))^2&=\sum_{i=1}^n\left(\log b_i-\frac {\pi^\R(B)}{\sqrt{n}}\right)^2\\
&=\sum_{i=1}^n(\log b_i)^2-\pi^\R(B)^2\leq (n-1)\pi^\R(B)^2,
\end{align*}
where the last inequality follows from the fact that, for positive numbers $x_i$, it holds $$\sum_{i=1}^n x_i^2\leq\left(\sum_{i=1}^n x_i\right)^2.$$
\end{proof}
In rank two we have the following stronger result:
\begin{cor}\label{cor:pos2}
If $n=2$, the matrix $B-A$ is positive definite if and only if $$\pi^\R(B)-\pi^\R(A)> \dSL(\pi^{\SL}(B),\pi^{\SL}(A)).$$
\end{cor}
\begin{proof} Given Lemma \ref{lem:pos}, we just need to show that if the condition on the projections holds, $B-A$ is positive definite, i.e.\ (by Lemma \ref{lem:eigenvalues}) that the eigenvalues of $BA^{-1}$ are strictly bigger than one. As in the proof of Lemma \ref{lem:pos}, we can reduce ourselves to the case where $A=\Id$, so we need to show that the eigenvalues $b_1\geq b_2$ of $B$ are bigger than one. But in this case we have
$$\pi^\R(B)-\pi^\R(S)=\frac{1}{\sqrt{2}}(\log b_1+\log b_2)$$
and
\begin{align*}
\dSL( \pi^{\SL}(B),\pi^{\SL}(A))&=\sqrt{\left(\log b_1-\frac{1}{2}(\log b_1 +\log b_2)\right)^2+\left(\log b_2-\frac{1}{2}(\log b_1 +\log b_2)\right)^2}\\
&=\frac{1}{\sqrt{2}}(\log b_1 -\log b_2)
\end{align*}
so the hypothesis implies $\log b_2>0$, i.e.\ $b_2>1$, and hence $b_1>1$ as well.
\end{proof}

\begin{remark}\label{dFandpiR}{ By restating Lemma \ref{Finslerdistance} using the notation of this section, we get that} given a pair of points $iX,iY\in\Yy_{0,l_\infty}$ such that $X-Y$ is positive definite,  the Finsler distance $\dF(X,Y)$ is, up to an explicit factor, equal to the difference of the projections of the two points on the Euclidean factor:   $$2d^F(iX,iY)=\sqrt n(\pi^\R(X)-\pi^\R(Y)).$$ 
This observation will play a crucial role in our proof of Theorem \ref{thm:main}. 
\end{remark}
\subsection{Orthogonality and orthogonal projection}\label{sec:orthogonality}
In this section we define the concept of orthogonality, following \cite{bp}.

\begin{definition}We say that two $\R$-tubes $\Yy_{a,b}$ and $\Yy_{c,d}$ are \emph{orthogonal} (and we write $\Yy_{a,b}\perp\Yy_{c,d}$) if they are orthogonal as submanifolds of the symmetric space endowed with the Riemannian metric.
\end{definition}
Concretely, $\Yy_{a,b}$ and $\Yy_{c,d}$ are orthogonal if they meet in a point $x$ and there their tangent spaces are orthogonal as subspaces of ${\rm T}_x\calX$ (with respect to the Riemannian metric). It was observed in \cite[Section 4.3]{bp} that the orthogonality relation can be expressed as a property of the cross-ratio of the boundary points: if $(a,c,b,d)$ is maximal, the $\R$-tubes $\Yy_{a,b}$ and $\Yy_{c,d}$ are orthogonal  if and only if $R(a,c,b,d)=2\Id$. 

Let us now fix a tube $\Yy_{a,b}$. In \cite[Section 4.3]{bp} the authors construct an involution $\sigma^{\Sp}_{a,b}$ of $\Sp(2n,\R)$, which is induced by the complex conjugation that fixes the real form $V_{a,b}$ of $\C^{2n}$ given by $$V_{a,b}=\langle v+iw|v\in a, w\in b\rangle.$$
We denote by  $\sigma_{a,b}\in\GL(2n,\R)$ the matrix corresponding to such linear map. Note that $\sigma_{a,b}$ is not a symplectic matrix, but for any two tubes $\Yy_{a,b}$ and $\Yy_{c,d}$ the product $\sigma_{a,b}\sigma_{c,d}$ belongs to $\Sp(2n,\R)$. 

The involution $\sigma^{\Sp}_{a,b}$ induces an anti-holomorphic map $\sigma_{a,b}^\calX$ of $\mathcal X$ whose fixed point set consists precisely of $\Yy_{a,b}$. It was verified in \cite[Cor. 4.7]{bp} that the $\R$-tubes orthogonal to $\Yy_{a,b}$ foliate the symmetric space $\Xx$. Moreover $\sigma_{a,b}$ induces also an involution
$$\sigma_{a,b}^\Ll:(\!(a,b)\!)\to (\!(b,a)\!)$$
where $(\!(a,b)\!):=\{c\in \Ll\, | \, (a,c,b) \mbox{ is maximal}\}$. 
For each Lagrangian $c$ such that $(a,c,b)$ is maximal, $\sigma^\Ll_{a,b}(c)$ is the unique Lagrangian $d$  with the property that $\Yy_{a,b}\perp\Yy_{c,d}$. 
Using these observations it is possible to define \cite[Cor. 4.7]{bp} the orthogonal projection
$$p_{a,b}:\Xx\cup (\!(a,b)\!)\to \Yy_{a,b}.$$

It will be useful to have concrete expression for the reflection $\sigma^{\Ll}_{a,b}$ and the restriction of orthogonal projection $p_{a,b}$ to $(\!(a,b)\!)$, for $(a,b)=(0,l_\infty)$. Recall that we identify $\Sym(n,\R)$ with the set of Lagrangians in $\Ll(\R^{2n})$ that are transverse to ${l_\infty}$ via the restriction of the affine chart $\iota:\Sym(n,\C)\to\Ll(\C^{2n})$. From the definitions, it follows that $\sigma_{0,l_\infty}=\bsm \Id_n&0\\0&-\Id_n\esm$. Further, one can also prove that:
\begin{lem}\label{perpto0infty}
For any $A\in \Sym(n,\R)$, the $\R$-tubes $\Yy_{A,-A}$ and $\Yy_{0,l_\infty}$ are orthogonal and their unique intersection point is $iA$. In particular $\sigma^\Ll_{0,l_\infty}(A)=-A$ and $p_{0,l_\infty}(A)=iA$. 
\end{lem}

We will need the fact that the vectorial distance of the projection of two Lagrangians $x,y$ to an $\R$-tube $\Yy_{a,b}$ can be computed in term of the eigenvalues of the cross-ratio of the four Lagrangians. 
\begin{lem}\label{R(a,x,y,b)}
If $(a,x,y,b)\in\Ll(\R^{2n})^{(4)}$ is a maximal 4-tuple  and $p_{a,b}$ is the orthogonal projection onto $\Yy_{a,b}$, the distance $\dC(p_{a,b}(x),p_{a,b}(y))$ is $(\log\mu_1,\dots,\log\mu_n)$, where the $\mu_i$ are the eigenvalues of $R(a,x,y,b)$.
\end{lem}
\begin{proof}
Up to the action of the symplectic group, we can reduce to the case $(a,b)=(0,l_\infty)$. In this case, the result follows from explicit computations.
\end{proof}
%
\section{Maximal representations}\label{maxreps}
As mentioned in the introduction, maximal representations are the representations that maximize the Toledo invariant, an invariant defined with the aid of bounded cohomology. It follows from a deep result of Burger, Iozzi and Wienhard that these representations can be equivalently characterized as representations admitting a ``well-behaved'' boundary map. Precisely, let $\G$ be the fundamental group of an oriented surface $\S$ with negative Euler characteristic and boundary $\partial \Sigma$ (which could be empty). Fix a finite area hyperbolization of $\S$ inducing an action of $\G$ on $\mathbb S^1=\partial \H^2$.
\begin{definition}\label{def:maxrep}
A representation $\rho:\G\to\Sp(2n,\R)$ \emph{admits a maximal framing} if there exists a $\rho$-equivariant map $\phi:\mathbb S^1\to\Ll(\R^{2n})$ which is monotone (i.e.\ the image of any positively oriented triple in the circle is a maximal triple) and right continuous.
\end{definition}
Burger, Iozzi and Wienhard (see \cite[Theorem 8]{biw}) proved that a representation admits a maximal framing if and only if it is maximal. Since we will not directly need bounded cohomology in the rest of the paper we refer the interested reader to \cite{biw} for a definition of the Toledo invariant.
  The following structural result about maximal representations allows us to associate, to each maximal representation, a locally symmetric space whose fundamental group is isomorphic to the fundamental group of $\S$. In the paper we will be interested in the geometry of such locally symmetric space.
\begin{theor}[{\cite[Theorem 5]{biw}}]
Maximal representations are  injective and have discrete image.
\end{theor}
Given a maximal representation $\rho:\Gamma\to \Sp(2n,\R)$, the image $\rho(\g)\in \Sp(2n,\R)$ of each non-peripheral element $\g\in\G$ is \emph{Shilov hyperbolic}: it is conjugate to $\bsm A&0\\0&^t A^{-1}\esm$ for a matrix $A$ in $\GL(n,\R)$ with all eigenvalues with absolute value greater than one (see \cite{Strubel}). Equivalently, $\rho(\gamma)$ fixes two transverse Lagrangians $\L^+_\g$ and $\L^-_\g$ on which it acts expandingly (resp.\ contractingly). Furthermore if $\phi:\mathbb S^1\to \Ll(\R^{2n})$ is the equivariant boundary map, the Lagrangians $\L^\pm_\g$ are the images $\phi(\g^\pm)$.\ \\

{\bf Assumption:} from now on, whenever a surface $\S$ has boundary, we will restrict to maximal representations $\rho$ such that the image of each peripheral element is Shilov hyperbolic. It is possible to show that this is equivalent to the requirement that the representation $\rho:{\Gamma}\to\Sp(2n,\R)$  is Anosov in the sense of \cite{GW}, and therefore we will denote maximal representations satisfying our assumption \emph{Anosov maximal representation}.
Moreover, we fix an orientation on the boundary components such that the surface lies \emph{to the right} of the boundary. This corresponds to a choice between a peripheral element and its inverse. We will always choose primitive peripheral elements according to the orientation of the boundary components.

\smallskip\smallskip\smallskip
The synthetic geometry whose lines are $\R$-tubes is particularly adapted to the study of maximal representations: if $\rho$ is a maximal representation, the boundary map $\phi$ allows us to select specific $\R$-tubes associated to the elements of the fundamental group. More precisely, for each element $\g\in\G$, $\rho(\g)$ stabilizes an $\R$-tube, denoted $\Yy_\g$, whose endpoints are the attractive and repulsive fixed points $\L_\g^\pm$ of $\rho(\g)$ in $\Ll(\R^{2n})$. It follows from Proposition \ref{intersectingtubes} that such tubes have the same intersection pattern as the axes of the elements in $\wt \S$. For ease of notation, if a tube is associated to an element $\gamma\in\Gamma$, we will denote the projection defined in section \ref{sec:orthogonality} by $p_\g$ and the associated involution by $\s_\g^{\Sp}$ (or $\sigma_{\rho(\gamma)}^{\Sp}$, if we want to underline the representation we are considering).
\smallskip\smallskip\smallskip

For points $x,y\in \mathbb{S}^1$, we denote by $(\!(x,y)\!)$ the subset of $\mathbb{S}^1$ given by points $z$ such that $(x,z,y)$ is a positive oriented triple. The following observation will be useful:
\begin{remark}\label{projcausal}
The transitivity of the $\Sp(2n,\R)$-action on the space of tubes and the identification of $\Yy_{0,l_\infty}$ with $\Xx_{\GL(n,\R)}$, allows us to generalize Definition \ref{def:Causal} and define the notion of a causal path $f:K\to \Yy$ for any tube $\Yy$. With this definition, if $\phi:\mathbb{S}^1\to \Ll(\R^{2n})$ is the boundary map of a maximal representation $\rho$, then for each $x<y$ the image of the map 
\begin{align*}
(\!(x,y)\!)&\to \Yy_{\phi(x),\phi(y)}\\
 t&\mapsto p_{\phi(x),\phi(y)}(\phi(t))
\end{align*}
is a causal path.
\end{remark}

\subsection{Some special maximal representations}
A special type of maximal representations we will be interested in are the ones obtained using the diagonal embedding $\Delta$ of $n$ identical copies of $\SL(2,\R)$ into $\Sp(2n,\R)$. The centralizer $Z_{\Sp(2n,\R)}(\Delta(\SL(2,\R)))$ is the subgroup, isomorphic to $\O(n)$, consisting of matrices of the form $\bsm A&0\\0&A\esm$, for matrices $A\in{\O}(n)$. 

Let $\rho':\G\to \PSL(2,\R)$ be the holonomy representation of a hyperbolic structure on $\S$. The representation $\rho'$ can be lifted to a maximal representation $\rho:\G\to\SL(2,\R)$. For any $\rho$ obtained this way and any character $\chi:\G\to Z_{\Sp(2n,\R)}(\Delta(\SL(2,\R)))$, the product $\Delta\circ\rho\times \chi$ is a maximal representation (see \cite{biw} for details) whose image is contained in $\Delta(\SL_2(\R))\times Z_{\Sp(2n,\R)}(\Delta(\SL(2,\R)))<\Sp(2n,\R)$. Observe that such a representation preserves the image of the diagonal inclusion of the Poincar\'e disk in the standard polydisk.
\begin{definition}
We say that a representation $\rho:\G\to \Sp(2n,\R)$ is \emph{a diagonal embedding of a hyperbolization} if there exists a lift $\rho_0:\G\to\SL(2,\R)$ of the holonomy of a hyperbolization  and a character  $\chi:\G\to Z_{\Sp(2n,\R)}(\Delta(\SL(2,\R)))$ such that $\rho$ is conjugated to $\Delta\circ\rho_0\times \chi$.
\end{definition} 

Another generalization of Teichm\"uller space that has attracted a lot of attention is the Hitchin component. If $i_N:\SL(2,\R)\to\SL(N,\R)$ denotes the unique irreducible representation, the \emph{Hitchin component} $\mbox{Hit}_N(\Gamma)$ is the component of the character variety $$\quot{\Hom(\G,\SL(N,\R))}{\SL(N,\R)}$$ containing  $i_N\circ \rho$, where $\rho$ the holonomy of a hyperbolization. Representations in the Hitchin component are called \emph{Hitchin representations}. When $N=2n$ is even, the representation $i_{2n}$ factors through $\Sp(2n,\R)$. It turns out that all the Hitchin representations whose image is contained in $\Sp(2n,\R)$ are also maximal representations \cite[Example 3.10]{bilw}. We call such representations   \emph{Hitchin maximal representations}.
\subsection{Distances and maximal representations}\label{sec:distandmaxreps}
Given $\gamma\in\Gamma$, we define the length  $\ellC (\rho(\gamma))$ (resp.\ $\ell^R(\rho(\gamma))$, $\ell^F(\rho(\gamma))$) of $\rho(\gamma)$ with respect to the vectorial (resp.\ Riemannian, Finsler) metric to be the translation length of $\rho(\gamma)$ acting on $\calX$ computed with the corresponding metric. 

Let $A$ be a matrix representing the action of $\rho(\g)$ on its attractive Lagrangian  $\L^+_\g$ and suppose $|a_1|\geq \dots \geq |a_n|>1$ are the absolute values of the eigenvalues of $A$. It is well known (see \cite{benoist_proprietes} and \cite{parreau_compactification}) that the vectorial length of $\rho(\gamma)$ is explicitely related to the eigenvalues of $A$:
$$\ellC(\rho(\gamma))=(2\log |a_1|,\dots,2\log |a_n|).$$
From this it is easy to deduce that
$$\ellF(\rho(\g))=\log \det A$$
and
$$\ellR(\rho(\g))=2\sqrt{\sum_{i=1}^n (\log |a_i|)^2}.$$
Note that the Finsler metric assigns to a peripheral element a length which is closely related to its action on the attractive Lagrangians. Moreover many points in $\Yy_\g$ realize the Finsler translation length of the element $\rho(\g)$:
\begin{lem}\label{lem:ellF}
Let $g\in \Sp(2n,\R)$ be Shilov hyperbolic with associated tube $\Yy$. For any $x\in\Yy$ such that $(x,g x)$ is a causal segment we have 
$$\ellF(g)=\dF(gx,x).$$
In particular, if $\rho$ is a maximal representation and $\phi$ is the associated boundary map, for any $x\in \mathbb{S}^1$, $$\ellF(\rho(\g))=\dF(p_\gamma(\phi(x)),p_\g(\phi(\g x))).$$
\end{lem}
\begin{proof}
Using the transitivity of the $\Sp(2n,\R)$-action we can assume that $\L_g^+=l_\infty$ and $\L_g^-=0$, so that $g$ has expression $\bsm A&0\\0&^t\!A^{-1}\esm$, and, up to conjugating $A$, we can assume that $x=i\Id$. In this case we get $g x=iA\,{^t\!A} $. The causality condition tells us that $A\,{^t\!A}$ is positive definite, and hence (by Lemma \ref{Finslerdistance}) $\dF(i\Id, iA\,{^t\!A})=\log\det A=\ellF(g)$. 
\end{proof}
Another important advantage of the Finsler metric is that it is additive on causal curves:
\begin{lem}\label{orsum}
Suppose $x,y,z\in (\!(\g^-,\g^+)\!)\subset S^1$ are positively oriented. Then
$$d_F(p_\g(\phi(x)),p_\g(\phi(z)))=d_F(p_\g(\phi(x)),p_\g(\phi(y)))+d_F(p_\g(\phi(y)),p_\g(\phi(z))).$$
\end{lem}
\begin{proof}
Up to the action of the symplectic group, we can assume $\phi(\g^-)=0$, $\phi(\g^+)=l_\infty$, $\phi(x)=\Id$, $\phi(y)=A$ and $\phi(z)=B$. By monotonicity, we know that $A-\Id$ and $B-A$ are positive definite. By Lemma \ref{perpto0infty}, the $\R$-tubes passing through $\Id$, $A$ and $B$ and orthogonal to $\Yy_{0,l_\infty}$ are $\Yy_{-\Id,\Id}$, $\Yy_{-A,A}$ and $\Yy_{-B,B}$ respectively, so $p_\g(\phi(x))=i\Id$, $p_\g(\phi(y))=iA$ and $p_\g(\phi(z))=iB$.

By Lemma \ref{Finslerdistance} 
\begin{align*}
\dF(p_\g(\phi(y)),p_\g(\phi(z)))&=\frac{1}{2}\log \det (A^{-1}B)\\
\dF(p_\g(\phi(x)),p_\g(\phi(y)))&=\frac{1}{2}\log \det A\\
\dF(p_\g(\phi(x)),p_\g(\phi(z)))&=\frac{1}{2}\log \det B,
\end{align*} which implies the desired equality. 
\end{proof}

The Finsler metric is also closely related to a cross-ratio in the sense of Labourie. Precisely, in \cite{Lab_cr} Labourie introduced a notion of $\R$-valued cross-ratio on the boundary at infinity of the fundamental group of a surface. In \cite{Lab_energy} he showed that
\begin{align*}
\B:\partial \Gamma^{4*}&\to \R\\
(x,y,z,t)&\mapsto \det R(\phi(x),\phi(t),\phi(y),\phi(z))
\end{align*}
is a cross-ratio in the sense of \cite{Lab_cr}, where
$$\partial \Gamma^{4*}=\{(x,y,z,t)\in\partial \Gamma^{4}\, |\, x\neq t \mbox{ and } y\neq z\}.$$ Moreover, given a cross-ratio $\B$, the period of a non-trivial element $\gamma\in\G$ is defined as
$$\ell_{\B}(\gamma):=\log | \B(\gamma^-,\gamma\cdot y,\gamma^+,y) |.$$
It is easy to check that
$$\ell^F(\gamma)=\frac{1}{2}\ell_{\B}(\gamma).$$
Labourie uses this cross-ratio in \cite{Lab_energy} to show that maximal representations are well displacing (cfr.\ also \cite{Har-Str} for a functorial extension to general Hermitian Lie groups). 
Recently Martone and Zhang used this language to prove systolic inequalities for the Finsler distance\footnote{Their result holds for the larger class of positively ratioed representations.} \cite{MZ}. 

As opposed to the Finsler metric, the Riemannian metric is not additive on causal paths. However we have the following (cfr.\ \cite[Lemma 9.3]{bp}):
\begin{lem}\label{ellRcausal}
Let $x_0,\ldots,x_k\in\Yy_\g$ be on a causal curve. Then 
$$\sum_{i=1}^k \dR(x_{i-1},x_i)\leq \sqrt n \dR(x_1,x_k).$$
\end{lem}

\subsection{Doubles}\label{sec:doubles}
Let $\Sigma$ be a surface with nonempty boundary. The purpose of this section is to construct, for each Anosov maximal representation $\rho:\pi_1(\Sigma,v)\to \Sp(2n,\R)$, what we call the \emph{holomorphic double} of $\rho$: a specific maximal representation $D\rho:\pi_1(D\Sigma,v)\to \Sp(2n,\R)$ of the fundamental group of the double of the surface $\S$ that restricts to the given representation $\rho$.

Denote by $c_0,\ldots,c_m$ the boundary components, and fix a basepoint $v$ on $c_0$. Recall that the double $D\Sigma$ of the surface $\Sigma$ is the surface obtained gluing two copies $\Sigma,\ov \Sigma$  of the surface $\Sigma$ along its boundary components (where $\ov \S$ is $\Sigma$ endowed with the opposite orientation). We denote by $j_0:\Sigma\to D\Sigma$ and $j_1:\ov\Sigma\to D\Sigma$ the natural inclusions and by $j:D\Sigma\to D\Sigma$ the involution fixing the boundary components pointwise and with the property that $j\circ j_0=j_1$. With a slight abuse of notation we will also denote by $j$ (resp.\ $j_i$) the maps induced at the level of fundamental groups. Given two paths $\alpha$ and $\beta$ such that $\alpha$ ends at the starting point of $\beta$, we denote by $\alpha\ast\beta$ their concatenation.

We denote by $c_i$ also loops parametrizing the respective boundary components and leaving the surface to the right. Recall that for any $\g\in\G$ we denote by $\s_{\rho(\g)}$ the element of $\GL(2n,\R)$ inducing the involution $\s_\g^{\Sp}$ (see Sections \ref{sec:orthogonality} and \ref{maxreps}). We set $\sigma_\rho:=\sigma_{\rho(c_0)}$.

Each arc $\alpha$ starting at $v$ and ending at a boundary component $c_i$ determines the peripheral element $\gamma_\alpha=\alpha\ast c_i\ast \alpha^{-1}$, which in turn gives a matrix $\sigma_{\rho(\gamma_\alpha)}\in\GL(2n,\R)$. 

Our holomorphic double will be obtained amalgamating the representation $\rho$ on $j_0(\pi_1(\Sigma,v))$ and  the $\sigma_\rho \rho \sigma_\rho$  on $j_1(\pi_1(\Sigma,v))$:
\begin{prop}\label{prop:double}
Let $\Sigma$ be a surface with nonempty boundary. Let $\rho:\pi_1(\Sigma,v)\to\Sp(2n,\R)$ be a maximal representation and assume that the image of every peripheral element is Shilov hyperbolic. Then $\rho$ is the restriction of a unique maximal representation $D\rho:\pi_1(D\Sigma,v)\to \Sp(2n,\R)$ such that: 
\begin{enumerate}
\item for all elements $\gamma$ in $\pi_1(D\Sigma,v)$, $D\rho(j(\gamma))=\sigma_\rho D\rho(\gamma) \sigma_\rho $;
\item for every arc $\alpha$ joining $v$ to a boundary component of $\Sigma$ we have
$$D\rho(\alpha\ast j(\alpha)^{-1})=\sigma_{\rho(\g_\alpha)}\sigma_\rho.$$
\end{enumerate}
\end{prop}
The double of a Hitchin representation was defined by Labourie and McShane in \cite[Section 9]{LMS}. The crucial difference of our setting is that in general maximal representations do not have diagonalizable image, as opposed to Hitchin representations. For this reason we cannot deduce our result from \cite{LMS} and we need to choose a different involution of $\Sp(2n,\R)$. As a result, if $\rho$ is a Hitchin maximal representation, the holomorphic double we define here is different from the Hitchin double defined in \cite{LMS}: if $i_{2n}$ is the irreducible representation of $\SL(2,\R)$ into $\Sp(2n,\R)$ and $h:\G\to \SL(2,\R)$ is the holonomy of a hyperbolization,  $Di_{2n}(h)$ is different from $i_{2n}(Dh)$. On the other hand, we have $\Delta(Dh)=D\Delta(h)$, and this motivated the choice of the definition of our double as a \emph{holomorphic} double.
\begin{proof}[Proof of Proposition \ref{prop:double}] 
To simplify the notation, we will drop the reference to the representation $\rho$ both in $\sigma_\rho$ and $\sigma_{\rho(\g_\alpha)}$, which we will simply denote by $\sigma$ and $\sigma_{\g_\alpha}.$ 

For all $i$ between $1$ and $m$, fix an arc $\alpha_i$ joining $v$ to the boundary component $c_i$. If $x_i$ denotes the concatenation $\alpha_i\ast j(\alpha_i)^{-1}$, and $\g_i\in\pi_1(\S,v) $ is the class of the concatenation $\g_i=\alpha_i\ast c_i\ast \alpha_i^{-1}$, a presentation for the  group $D\G$ is given by 
$$\pi_1(D\Sigma,v)=\left\langle  j_0(\pi_1(\Sigma,v)),j_1(\pi_1(\Sigma,v)), x_1,\ldots,x_m\left|\,j_0(c_0)j_1(c_0)^{-1}, j_0(\gamma_i)^{-1}x_ij_1(\gamma_i)x_i^{-1}\right.\right\rangle.$$
In particular, there exists at most one representation $D\rho:D\G\to \Sp(2n,\R)$ satisfying the hypotheses of Proposition \ref{prop:double}, because the requirements of Proposition \ref{prop:double} uniquely determine the image of $D\rho$ on the generators. Indeed, we need to set for all $\g\in \pi_1(\Sigma,v)$
\begin{equation}\label{DrhoSigma}
D\rho(j_0(\g))=\rho(\gamma)
\end{equation}
\begin{equation}\label{Drhoj1Sigma}
D\rho(j_1(\g))=\sigma\rho(\gamma)\sigma
\end{equation}
and for all $i\in\{1,\dots, m\}$
\begin{equation}\label{Drhoxi}
D\rho(x_i)=\sigma_{\g_i}\sigma.
\end{equation}
To show that $D\rho$ exists, we just need to show that (\ref{Drhoxi}), (\ref{DrhoSigma}) and (\ref{Drhoj1Sigma}) determine a well defined maximal representation. To prove that it is well defined we need to check that:
\begin{itemize}
\item property (1) holds for the $x_i$;
\item the relations are mapped to the identity by $D\rho$;
\item for any $\alpha\ast j(\alpha)$ different from $x_i$, (2) holds.
\end{itemize}
For the first point, we have
\begin{align*}
D\rho(j(x_i))&=D\rho(j(\alpha_i)\ast \alpha_i^{-1})=D\rho((\alpha_i\ast j(\alpha_i)^{-1})^{-1})\\
&=[D\rho(\alpha_i\ast j(\alpha_i)^{-1})]^{-1}\stackrel{(\ref{Drhoxi})}{=}(\sigma_{\g_i}\sigma)^{-1}\\
&=\sigma\sigma_{\g_i}\sigma \sigma\stackrel{(\ref{Drhoxi})}{=}\sigma D\rho(x_i)\sigma.
\end{align*}
For the first relation, since $\sigma$ and $\rho(c_0)$ commute, we have:
\begin{align*}
D\rho(j_0(c_0)j_1(c_0)^{-1})&=\left(D\rho(j_0(c_0)\right)\left(D\rho(j(j_0(c_0)^{-1})\right)\\&\stackrel{(\ref{DrhoSigma}),(\ref{Drhoj1Sigma})}{=}\rho(c_0)\sigma \rho(c_0)^{-1}\sigma=\Id.
\end{align*}
For the other relations, we use the fact that $\sigma_{\gamma_i}$ commutes with $\rho(\gamma_i)$. This implies:
$$
D\rho( j_0(\gamma_i)^{-1}x_ij_1(\gamma_i)x_i^{-1}) \stackrel{(\ref{DrhoSigma}),    (\ref{Drhoj1Sigma}), (\ref{Drhoxi})}{=}(\rho(\gamma_i^{-1}))(\sigma_{\gamma_i}\sigma)(\sigma \rho(\gamma_i)\sigma)(\sigma \sigma_{\gamma_i})=\Id.
$$
Let now $\alpha$ be any other arc with endpoint in the component $c_i$. Denote by $\delta_\alpha$ the concatenation $\delta_\alpha=\alpha\ast\alpha_i^{-1}$, so that $\alpha=\delta_\alpha\ast\alpha_i$. We get:
\begin{align*}
D\rho(\alpha\ast j(\alpha)^{-1})&=D\rho(\delta_\alpha(\alpha_i\ast j(\alpha_i)^{-1})j(\delta_\alpha)^{-1} )\\
&\stackrel{(\ref{DrhoSigma}),(\ref{Drhoxi})}{=}\rho(\delta_\alpha)\sigma_{\gamma_i}\sigma\sigma\rho(\delta_\alpha)^{-1}\sigma
\end{align*}
and since one can verify that
$$\rho(\delta_\alpha)\sigma_{\gamma_i}\rho(\delta_\alpha)^{-1}=\sigma_{\gamma_\alpha}$$
we deduce that
$$D\rho(\alpha\ast j(\alpha)^{-1})=\sigma_{\gamma_\alpha}\sigma,$$
as required.

The fact that $D\rho$ is a maximal representation follows from the additivity formula for the Toledo invariant (see \cite[Theorem 1(3)]{biw}, and also \cite{Strubel}) and the fact that the restriction of $D\rho$ to $j_1(\pi_1(\S,v))$ is maximal being conjugate via an anti-holomorphic isometry to a maximal representation of a surface with the opposite orientation.
\end{proof}

For how we defined it, the double of a representation depends on the choice of a boundary component and of a base point on it. The goal of the next proposition is to show that we can forget about this choice, as up to conjugation in the fundamental group we get the same double.

\begin{prop}\label{prop:samedouble}
Let $v,w\in \partial\Sigma$ and $\beta$ an arc from $v$ to $w$. Let $\eta:\pi_1(\Sigma,w)\to\Sp(2n,\R)$ be an Anosov maximal representation
. Then
$$D\eta(\delta)=D\rho(\beta\ast \delta\ast\beta^{-1}),$$
where $\rho:\pi_1(\Sigma,v)\to\Sp(2n,\R)$ is the representation given by $$\rho(\gamma)=\eta(\beta^{-1}\ast \gamma\ast\beta).$$
\end{prop}
\begin{proof} Denote by $c$ be boundary component containing $v$ and by $d$ the one containing $w$.

Note first that since $\eta$ is maximal, $\rho$ is maximal as well, so we can define its double as in Proposition \ref{prop:double}. To show the equality, we will prove that $\theta(\gamma):=D\rho(\beta\ast \gamma\ast\beta^{-1})$ is a maximal representation which extends $\eta$ and satisfies the two conditions of Proposition \ref{prop:double}. By uniqueness, this will imply that $\theta=D\eta$.

Clearly $\theta$ is maximal, since $D\rho$ is. Suppose $\delta\in \pi_1(\Sigma,w)$. Then $\beta\ast \delta\ast\beta^{-1}\in \pi_1(\Sigma,v)$, so
$$\theta(j_0(\delta))=D\rho(j_0(\beta\ast \delta\ast\beta^{-1}))=\rho(\beta\ast \delta\ast\beta^{-1})=\eta(\delta),$$
i.e.\ $\theta$ extends $\eta$.

We now prove property (1). We have:
\begin{align*}
\theta(j(\delta))&=D\rho(\beta\ast j(\delta)\ast \beta^{-1})=D\rho(\beta\ast j(\beta)^{-1})D\rho(j(\beta\ast\delta\ast \beta^{-1}))D\rho(j(\beta)\ast\beta^{-1})\\
&=\s_{\rho(\g_\beta)}\sigma_\rho\sigma_\rho\theta(\delta)\sigma_\rho\sigma_\rho\s_{\rho(\g_\beta)}
\end{align*}
and since $\rho(\g_\beta)=\rho(\beta\ast d\ast\beta^{-1})=\eta(d)$, we have $\s_{\rho(\g_\beta)}=\sigma_\eta$, which shows (1).

Finally, we show that property (2) holds. Let $\alpha$ be an arc joining $w$ with a boundary component $e$. Then
\begin{align*}
\theta(\alpha\ast j(\alpha)^{-1})&=D\rho(\beta\ast\alpha\ast j(\alpha)^{-1}\ast\beta^{-1})\\
&=D\rho((\beta\ast\alpha)\ast j(\beta\ast\alpha)^{-1})D\rho(j(\beta)\ast\beta^{-1})\\
&=\sigma_{\rho(\g_{\beta\ast\alpha})}\sigma_\rho\sigma_\rho\sigma_{\rho(\g_\beta)}.
\end{align*}
Since we have $$\rho(\g_{\beta\ast\alpha})=\rho(\beta\ast\alpha\ast e\ast \alpha^{-1}\ast\beta^{-1})=\eta(\g_\alpha)$$
and we know from before that $\sigma_{\rho(\g_\beta)}=\sigma_\eta$, we get $\theta(\alpha\ast j(\alpha)^{-1})=\sigma_{\eta(\gamma_\alpha)}\sigma_\eta$.
\end{proof}
If a surface has nonempty boundary, the boundary map associated to a maximal representation of its fundamental group is only right-continuous in general, while for closed surfaces the boundary map has especially good properties:
\begin{theor}[{\cite[Corollary 6.3]{bilw}}]
Let $\rho :\pi_1(\Sigma) \to \Sp(2n,\R )$ be a maximal representation, where $\Sigma$ is a closed surface. Then there is a $\rho$-equivariant continuous injective map $\phi :\mathbb  S^ 1 \to \Ll(\R^{2n} )$ with rectifiable image.
\end{theor}

 We will associate to a maximal representation $\rho$ of a surface with boundary the boundary map of $D\rho$, so that we have a continuous map in this case as well. 
 More precisely, we denote by $D\G$ the fundamental group of the double of $\S$ and fix once and for all an action $h$ of $D\G$ on $\H^2$, inducing an action of $D\G$ on $\mathbb{S}^1$.
 Let $\phi:\mathbb S^1\to\Ll(\R^{2n})$ be the continuous boundary map associated to  $D\rho$. We will denote by $\L(\G)\subset \mathbb S^1$ the limit set of $\G$ on $\partial \H^2$. If we fix any finite-area hyperbolization of $\Sigma$ with holonomy $h_0$ and denote by $t:\mathbb S^1\to \L(\G)$ the $(h_0,h|_{\G})$-equivariant map associating to any parabolic point $p$ for $h_0$ the attractive fixed point of $h(p)$, then the composition $\phi\circ t$ is the equivariant boundary map of Definition \ref{def:maxrep}\footnote{Cfr. also \cite{Nico-JP} for a different explicit construction of the boundary map.}.

\section{Defining orthotubes}\label{sec:ort}
The purpose of the section is to define a notion of orthogeodesic in our setting, which will allow us to give a geometric interpretation of our main result as inequalities relating intrinsic geometric quantities in the locally symmetric space associated to a maximal representation. We will define orthotubes and their length, and we will prove that the length of an orthotube is half the length of the corresponding curve in the double of the surface.


\subsection{The classical definition}
Let $\Sigma$ be a hyperbolic surface with nonempty geodesic boundary and fundamental group $\Gamma$. Fix a boundary component $c$ and consider the set of \emph{oriented} orthogeodesics starting from $c$, that is, the set of oriented geodesic segments with first endpoint on $c$, second endpoint in $\partial\Sigma$ and orthogonal to the boundary at both endpoints. If we fix a fundamental domain $I$ in a lift $\tilde{c}$ of $c$ in $\tilde{\Sigma}=\hyp$, any oriented orthogeodesic starting from $c$ can be lifted uniquely to a geodesic segment in $\H^2$  orthogonal to $I$ and to the lift $\tilde{d}$ of a boundary component $d$ ($c$ and $d$ can be the same, but $\tilde{c}$ and $\tilde{d}$ are different). 

On the one hand, this implies that the set of orthogeodesics is in bijection with the set of arcs starting at $c$ and ending at a boundary component modulo homotopy relative to the boundary of $\Sigma$. On the other, the set of oriented orthogeodesics starting from $c$ is also in bijection with the set
$$\quot{\{\tilde{d}\subseteq{\partial\widetilde\Sigma}, \tilde{d}\neq \tilde{c}\}}{\mbox{Stab}_\Gamma(\tilde{c})}.$$
Now, each lift of a boundary component corresponds to a unique primitive peripheral element of $\G$ with the correct orientation\footnote{From now on \emph{peripheral element} will mean primitive peripheral element with compatible orientation.}. So if we fix the peripheral element $\gamma$ corresponding to $c$, the set of oriented orthogeodesics starting from $c$ is also in bijection with
$$ \ort^{\hyp}_\Sigma(\gamma):=\quot{\{\delta\neq\g\;\mbox{peripheral}\}}{\langle\gamma\rangle},$$
where $\langle\gamma\rangle$ acts by conjugation. 

Similarly, if $\ort^{\hyp}_\Sigma$ denotes the set of all \emph{unoriented} orthogeodesic in $\S$, we can see that it is in bijection with
$$\ort^{\hyp}_\Sigma:=\quot{\{\{\gamma,\delta\}\,|\,\gamma\neq \delta\;\mbox{peripheral}\}}{\G}.$$
\subsection{Orthotubes in $\bm{\calX}$}\label{sec:orthogeodesics}
Interpreting orthogeodesics as a pairs of peripheral elements gives a natural way to define orthotubes as orthogonal $\R$-tubes in $\calX$:
\begin{definition}An \emph{orthotube} corresponding to the pair $\alpha=(\gamma,\delta)$ is a tube $\Yy_\alpha$ in $\calX$ orthogonal to $\Yy_\g$ and $\Yy_\delta$.
\end{definition}
It is easy to show that for every pair of peripheral elements there exists a unique orthotube:
\begin{lem}\label{existenceorth}
For each pair of peripheral elements, an orthotube exists and is unique. The assignment of orthotubes is $\G$-equivariant.
\end{lem}
\begin{proof}
Consider $\gamma$ and $\delta$; by \cite[Proposition 2.11]{bp} we can assume that $\phi(\g^+)=-\Id$, $\phi(\g^-)=\Lambda$, $\phi(\delta^+)=0$ and $\phi(\delta^-)=l_\infty$, for some diagonal matrix $\Lambda$ with eigenvalues between $-1$ and $0$. If $\Yy$ is orthogonal to $\Yy_\delta$, it must then be of the form $\Yy_{-A,A}$ for some positive definite symmetric matrix $A$. We show that there exists a unique matrix $A$ such that $\Yy$ is orthogonal to $\Yy_\gamma$ as well, i.e.\ such that
$$R(-\Id,-A,\Lambda,A)=2\Id.$$
By explicit computations using \eqref{cross-ratio}, this is equivalent to $A^2=-\Lambda$, which has indeed one and only one positive definite solution.
\end{proof}
We define the length of an orthotube $\Yy_\alpha$ to be the distance, with respect to the vectorial,  Riemannian or Finsler metric, of the unique intersection points with the tubes associated with the peripheral elements determining $\alpha$. We denote such distances by $\ellC(\alpha)$, $\ellR(\alpha)$ and $\ellF(\alpha)$ respectively. Observe that any Riemannian geodesic segment in $\calX$ with endpoints in $\Yy_\gamma$ and $\Yy_\delta$ and orthogonal to both tubes is necessarily contained in $\Yy_\alpha$ and has length $\ellR(\alpha)$. In particular in the locally symmetric space $\G\backslash \calX$ the Riemannian length of an orthotube is the length of a local length minimizer between the projections of two peripheral tubes.

In analogy with the hyperbolic case, we define $\ort_\Sigma(\gamma)$ to be the set of all orthotubes associated to pairs $(\g,\delta)$, for every peripheral element $\delta\neq \gamma$, up to the action of $\langle\gamma\rangle$ by conjugation, and $\ort_\Sigma$ to be the union of all $\ort_\Sigma(\gamma)$ up to the action of $\G$ by conjugation and of $\mathbb Z/2\mathbb Z$ by switching endpoints. Note that any two orthotubes in the same class have the same length, so we can talk about the length of an element of $\ort_\Sigma(\gamma)$ or $\ort_\Sigma$.
\begin{remark}\label{sumorthogeodesics}
Each element $\alpha\in\ort_\Sigma$ appears twice in the union of all $\ort_\Sigma(\gamma)$. So for any positive real-valued function $f$, if $\Sigma$ is a surface with $m$ boundary components represented by $\g_1,\dots,\g_n$, we have:
$$2\sum_{\alpha\in\ort_\Sigma}f(\ellC(\alpha))=\sum_{i=1}^{m}\sum_{\alpha\in\ort_\Sigma(\gamma_i)}f(\ellC(\alpha)).$$
\end{remark}
\subsection{Orthotubes and doubles}
In the case of hyperbolic surfaces, any orthogeodesic doubles to a closed geodesic in the double of the surface, whose length is twice the length of the orthogeodesic. The purpose of the section is to show that the same holds in our setting. Formally  if $\alpha=(\delta_1,\delta_2)$ is an orthotube of $\Sigma$, it corresponds to a homotopy class of paths, also denoted by $\alpha$ with a slight abuse of notation, between two boundary components of $\Sigma$, and hence to an element $D\alpha\in\pi_1(D\Sigma,v)$. Explicitly, if $\ov\beta$ denotes a path in $\wt \Sigma$ between the preferred lift of $v$ and the axis of $\delta_1$, and if $\beta$ is the projection to $\Sigma$ of $\ov \beta$, $D\alpha$ is the class, in $\pi_1(\S,v)$ corresponding to $\beta\ast\alpha\ast j(\alpha)^{-1}\ast\beta^{-1}$.  We have the following:
\begin{prop}\label{prop:4.6}
For each orthotube $\alpha$ we have 
\begin{enumerate}
\item $\Yy_{\alpha}=\Yy_{D\alpha};$
\item $2\ellF(\alpha)=\ellF(D\alpha).$
\end{enumerate}
\end{prop}
\begin{proof}
We show the result for orthotubes $\alpha$ starting from the boundary component $c_0$. The general result will follow from Proposition \ref{prop:samedouble}.

Since $D\alpha=\alpha\ast j(\alpha)^{-1}$, we know that $D\rho(D\alpha)=\sigma_{\gamma_\alpha}\sigma_\rho$. Note that by \cite[Lemma  4.15]{bp}, if $\Yy_{a,b}\perp \Yy_{c,d}$, then setwise $\sigma_{a,b}\Yy_{c,d}=\Yy_{c,d}$, and by \cite[Lemma  4.11]{bp} $\sigma_{a,b}\Yy_{a,b}=\Yy_{a,b}$. As a consequence, $$D\rho(D\alpha)\Yy_\alpha=\Yy_\alpha.$$
Suppose the endpoints of $\Yy_\alpha$ are Lagrangians $a,b$, where $(\Lambda_{\gamma_\alpha}^{ +},a,\Lambda_{\gamma_\alpha}^{ -},b)$ is maximal. We want to show that $b$ (resp.\ $a$) is the repulsive (resp.\ attractive) Lagrangian of $D\rho(D\alpha)$. Note first that since $D\rho(D\alpha)$ fixes the tube $\Yy_{\alpha}$, it either fixes or exchanges $a$ and $b$. But
$$D\rho(D\alpha)(a)=\sigma_{\gamma_\alpha}\sigma_\rho (a)=\sigma_{\gamma_\alpha}(b),$$
where the second equality holds by \cite[Lemma 4.15]{bp}. But since $b$ is not an endpoint of $\Yy_{\gamma_\alpha}$, it is not fixed by $\sigma_{\gamma_\alpha}$, which implies that $D\rho(D\alpha)(a)=a$ and $D\rho(D\alpha)(b)=b$. 

Moreover, since $D\rho(D\alpha)^{-1}\Lambda_{\gamma_\alpha}^-=\sigma_\rho\Lambda^-_{\gamma_\alpha}$, we get that $(a,\Lambda_{\gamma_\alpha}^-,D\rho(D\alpha)^{-1}\Lambda_{\gamma_\alpha}^-,b)$ is maximal. We claim that this implies that $b$ and $a$ are the repulsive and attractive Lagrangians of $D\rho(D\alpha)$. This follows from the following observation:
\begin{Remark} Let $g\in\Sp(2n,\R)$ be Shilov hyperbolic fixing two Lagrangians $a$ and $b$. If there exists a point $x\in(\!(a,b)\!)$ with $(a,x,g^{-1}x,b)$ maximal, then $a=\L_g^+$ and $b=\L_g^-$.
\end{Remark}
To prove the remark, observe that up to the symplectic group action we can assume that $(a,x,b)=(l_\infty,\Id,0)$ and $g=\bsm A&0\\0&^tA^{-1}\esm$, and we need to verify that $l_\infty$ is the attractive Lagrangian for $g$, namely that the eigenvalues of $A$ are precisely the eigenvalues of $g$ that have absolute value bigger than 1. But the hypothesis that  $(a,x,g^{-1}x,b)$ is maximal implies that $(a,gx,x,b)$ is maximal and so
$$1<\min\ev(A ^tA)\leq |\min\ev (A)|^2,$$
as requested.

So we get $\Yy_\alpha=\Yy_{D\alpha}$.
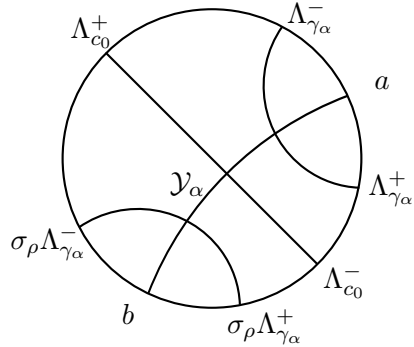
\begin{figure}[h]
\begin{overpic}{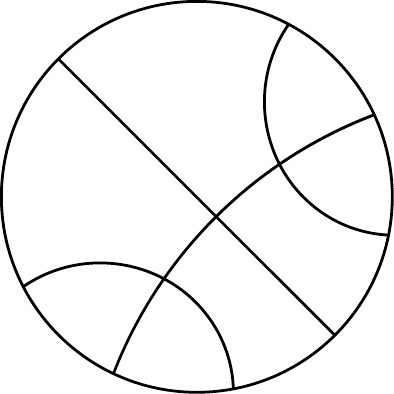}
\put(-17,20){$\sigma_\rho\Lambda_{\g_\alpha}^-$}
\put(55,-8){$\sigma_\rho\Lambda_{\g_\alpha}^+$}
\put(102,38){$\Lambda_{\g_\alpha}^+$}
\put(75,95){$\Lambda_{\g_\alpha}^-$}
\put(36,39){$\Yy_{\alpha}$}
\put(87,7){$\Lambda_{c_0}^-$}
\put(3,90){$\Lambda_{c_0}^+$}
\put(20,-5){$b$}
\put(104,73){$a$}
\end{overpic}
\caption{The $\R$-tubes appearing in the proof}\label{fig:doubletube}
\end{figure}

We now want to show the statement about the Finsler lengths. Denoting the projection onto $\Yy_\alpha$ by $p_\alpha$, we have
\begin{align*}
\ellF(\alpha)&=\dF(p_\alpha(\Lambda_{\gamma_\alpha}^+),p_\alpha(\Lambda_{c_0}^-))\\
&=\frac{1}{2}\left(\dF(p_\alpha(\Lambda_{\gamma_\alpha}^+),p_\alpha(\Lambda_{c_0}^-))+\dF(p_\alpha(\underbrace{\sigma_\rho(\Lambda_{c_0}^-)}_{\Lambda_{c_0}^-}),p_\alpha(\sigma_\rho(\Lambda_{\gamma_\alpha}^+))\right)\\
&=\frac{1}{2}\dF(p_\alpha(\Lambda_{\gamma_\alpha}^+),p_\alpha(\sigma_\rho(\Lambda_{\gamma_\alpha}^+))
\end{align*}
where the last equality follows from the additivity of the Finsler metric on causal curves (Lemma \ref{orsum}). But $D\rho(D\alpha)^{-1}(\Lambda_{\gamma_\alpha}^+)=\sigma_\rho\Lambda_{\gamma_\alpha}^+$, thus by Lemma \ref{lem:ellF}
$$\dF(p_\alpha(\Lambda_{\gamma_\alpha}^+),p_\alpha(\sigma_\rho(\Lambda_{\gamma_\alpha}^+))=\ellF(D\alpha),$$
which implies that $\ellF(\alpha)=\frac{1}{2}\ellF(D\alpha)$.

\end{proof}


\section{Strategy of proof}\label{sec:strategy}
The idea of the proof of Basmajian's identity is the following. Fix a lift $\tilde{c}$ of a boundary component $c$ to $\hyp$, with endpoints $x$ and $y$. Pick $z\in\tilde{c}$ and let $\g$ be the peripheral element with axis $\tilde{c}$. We can write $(z,\g z)$ as
$$(z,\g z)=\left(p_{\tilde{c}}(\, (\!(x,y)\!)\cap \Lambda(\Gamma))\cup\!\!\!\!\!\!\!\!\!\bigcup_{\substack{\tilde{d}\subset \partial \wt\Sigma\\\mbox{\tiny with endpoints}\\ \mbox{\tiny in }(\!(x,y)\!)}}\!\!\!\!\!\!\!\!\!p_{\tilde{c}}(\tilde{d})\right)\cap (z,\g z)$$
where $p_{\tilde{c}}$ is the orthogonal projections onto $\tilde{c}$ and $\Lambda(\Gamma)$ is the limit set of $\Gamma$. For every other component $\tilde{d}$ of $\partial\tilde{\Sigma}$, we can compute the length of the projection onto $\tilde{c}$ in terms of the length of the corresponding orthogeodesic, using hyperbolic trigonometry. Since the limit set has measure zero (see Proposition \ref{prop:limitset}), we can deduce that its projection onto $\tilde{c}$ has measure zero as well. The length of $(z,\g z)$ is the length of $c$, and hence we obtain Basmajian's identity.
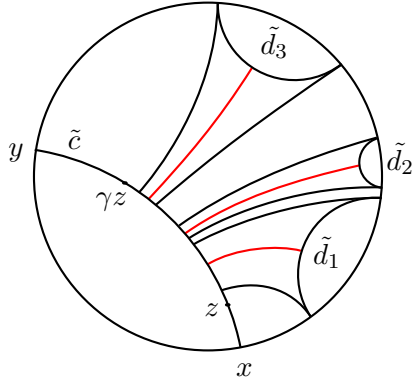
\begin{figure}[h]
\vspace{.5cm}
\begin{overpic}{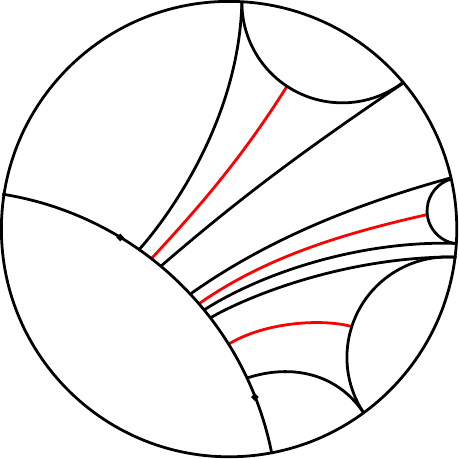}
\put(10,58){$\tilde{c}$}
\put(65,85){$\tilde{d_3}$}
\put(101,52){$\tilde{d_2}$}
\put(80,25){$\tilde{d_1}$}
\put(58,-7){$x$}
\put(-7,56){$y$}
\put(49,10){$z$}
\put(18,42){$\g z$}
\end{overpic}
\caption{Some projections on the lift $\tilde{c}$ (orthogonals in red)}
\end{figure}

Our proof of \eqref{Finsler}, Basmajian-type inequalities for the Finsler metric, follows the same strategy of the classical proof, using additivity of the Finsler metric along causal paths (Lemma \ref{orsum}) and the important fact that the Finlser translation length of an element is attained at any point along a causal curve (Lemma \ref{lem:ellF}). With this at hand, given any peripheral element $\gamma$, we look at the $\R$-tubes corresponding to the other peripheral elements $\delta$ and we compute the vectorial distance between the projection of the two endpoints of a tube $\Yy_\delta$ onto $\Yy_\g$ in terms of the vectorial length of the orthogeodesic between them (Lemma \ref{2logcothC}), which in turn (Lemma \ref{2logcoth}) gives an inequality for the Finsler metric.  The generalization (Theorem \ref{thm:M0}) of the fact that the limit set has measure zero is proven in Section \ref{sec:0measure}, following the strategy explained in Section \ref{classicalmeasure0}.

The characterization of diagonal embeddings as representations attaining the equalities in \eqref{Finsler} follows from the observation that having equalities is equivalent (Lemma \ref{2logcothC}) to the fact that the cross-ratios of the form $R(\L_\g^-,\L_\delta^+,\L_\delta^-,\L_\g^+)$ are multiples of the identity. By Lemma \ref{lem:iff}, this implies that the image of the boundary map is in the boundary of a diagonal disk, which allows us to deduce the characterization we want.

The Riemannian case is a priori harder, since neither Lemma \ref{lem:ellF} nor Lemma \ref{orsum} hold true for the Riemannian distance. However we rely on the observation that the translation length of a peripheral element $\g$ is at most the translation length of $\g$ on the Euclidean factor of $\Yy_\g$, and the latter quantity is, up to a constant, its Finsler translation length (Remark \ref{dFandpiR}). This allows us to deduce \eqref{Riemannian} as a consequence of the result about the Finsler metric.
\section{An identity between cross-ratios}\label{sec:0measure}
Basmajian's proof of his celebrated identity builds on the following well known fact:
\begin{prop}\label{prop:limitset}
Let $\Gamma$ be a Fuchsian group corresponding to the holonomy of a hyperbolic surface $\Sigma$ with nonempty geodesic boundary. Then the Lebesque measure of the limit set $\Lambda=\Lambda(\Gamma)$ is zero.
\end{prop}
The goal of this section is to show that the analogous result holds for Anosov maximal representation, by adapting to the higher rank setting a proof which has been kindly suggested to us by Brian Bowditch (a similar argument can be found in \cite{Tukia}). 
 
We will use the following result, giving a sufficient condition for a subset of an interval to have Lebesgue measure zero.
\begin{lem}\label{criterionmeasure0}
Let $[0,\ell]$ be an interval in $\R$, $X$ be a union of open subintervals $\{I_\alpha\}_\alpha$ of $[0,\ell]$ and $Y=[0,\ell]\setminus X$. Suppose there exists  $\lambda>0$ such that for every interval $I\subset[0,\ell]$ there exists a finite union of closed intervals $J\subset I$ such that
\begin{enumerate}
\item $J\subset X\cap I$, and
\item $\lm(J)\geq \l\lm(I)$.
\end{enumerate}
Then $\lm(Y)=0$.
\end{lem}
\begin{proof}
We construct by induction a sequence of sets $Y_k\supset Y$ which is a union of intervals and such that $\lm(Y_k)\leq (1-\l)^k\ell$.

For $k=0$, consider $I=[0,\ell]$. By hypothesis, there exists a finite union of subintervals $J_0\subset X$ such that $\lm(J_0)\geq \l \ell$. Set $Y_0:=X\setminus J_0$.
Suppose we have constructed $Y_k=\bigcup_\beta I^k_\beta$ with $\lm(Y_k)\leq (1-\l)^k\ell$. Then for each $I^k_\beta$ we know that there exists a finite union of subintervals $J^k_\beta$ satisfying (1) and (2); we set $Y_{k+1}:=Y_k\setminus \bigcup_\beta J^k_\beta$. Then $Y_{k+1}$ is still a union of intervals and
$$\lm(Y_{k+1})\leq (1-\l)\lm(Y_k)\leq (1-\l)^{k+1}\ell.$$
\end{proof}
\subsection{The classical proof}\label{classicalmeasure0}
In this section we sketch the proof of Proposition \ref{prop:limitset}.

We will show that for any compact interval $T\subset \partial\hyp\simeq\mathbb{S}^1$, $\lm(T\cap \Lambda)=0$. 
Note that the measure $\lm$ on $T$ is the one induced by the measure on $\mathbb{S}^1$. Fix an identification of $\partial\hyp$ with $\R\cup\{\infty\}$ so that $T=[0,L]$. Since $(T,\lm)$ is $C$-biLipschitz to $T$ endowed with the distance $|\cdot|$ induced by $\R$, it is enough to prove our claim using $|\cdot|$.

We want to apply Lemma \ref{criterionmeasure0} and we claim that it is enough to verify the assumptions only for subintervals $I=[a,b]\subset T$ whose extrema  $a,b$ belong to $\Lambda(\G)$. Indeed, consider $a'=\inf\{x\in[a,b] | x\in\Lambda\}$ and $b'=\sup\{x\in[a,b] | x\in\Lambda\}$. We have $I=[a,a')\cup [a',b']\cup(b',b]$ and the first and the last subintervals are contained in $I\setminus \Lambda$.

We will use the following observation:
\begin{Remark}\label{distboundary} If $\Sigma$ is a complete hyperbolic structure on a compact surface with non-empty boundary, there is a constant $r=r(\Sigma)$ such that for every $p\in \Sigma$, the distance of $p$ from the boundary is at most $r$.
\end{Remark}
Consider the unique point $p$ in the geodesic between $a$ and $b$ with imaginary part $(b-a)/2$. Note that $p\in\wt{\Sigma}$ because $a$ and $b$ are in the limit set. By Remark \ref{distboundary}, there is a point $q\in\partial\wt{\Sigma}$ at distance at most $r$ from $p$, which implies that $\im q\geq e^{-r}(b-a)/2$. Moreover $q$ belongs to a lift of a boundary component of $\S$ with endpoints $c<d$ in $I$, so $(c,d)\subset I\setminus \Lambda$, and we have
$$(d-c)\geq 2\im q \geq e^{-r}(b-a).$$
So we can set $\lambda=e^{-r}$ and apply Lemma \ref{criterionmeasure0} to deduce that $|\Lambda\cap T|=0$.
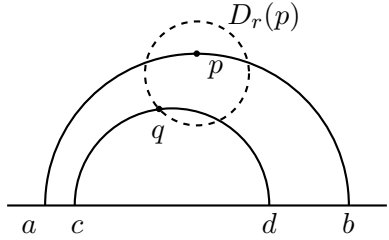
\begin{figure}[H]
\begin{overpic}{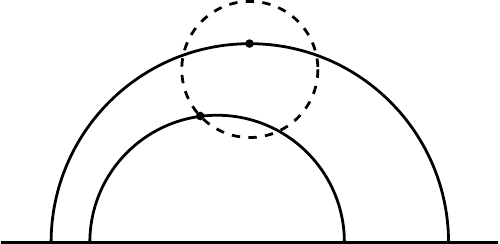}
\put(4,-7){$a$}
\put(17,-7){$c$}
\put(67,-7){$d$}
\put(88,-7){$b$}
\put(38,18){$q$}

\put(53,35){$p$}
\put(58,48){$D_{r}(p)$}
\end{overpic}
\caption{Finding the subinterval}
\end{figure}
\subsection{Maximal representations}\label{sec:m0max}
Our objective is to prove the following result: 
\begin{theor}\label{thm:M0}
Let $\Sigma$ be a surface with nonempty boundary and let $\rho:\pi_1(\Sigma)\to\Sp(2n,\R)$ be  an Anosov maximal representation. For every peripheral element $\gamma$ we have:
$$\ellF(\g)=\sum_{\alpha\in\ort_\Sigma(\gamma)} \dF(p_\g(\phi(\delta_\alpha^+)),p_\g(\phi(\delta_\alpha^-)))$$
where $\phi$ is the boundary map associated to $\rho$ and $\delta_\alpha$ is the peripheral element associated to $\alpha$ and different from $\gamma$.
\end{theor}
\begin{Remark}\label{rem:parabolic}
 Observe that the inequality 
 $$\ellF(\g)\geq\sum_{\alpha\in\ort_\Sigma(\gamma)} \dF(p_\g(\phi(\delta_\alpha^+)),p_\g(\phi(\delta_\alpha^-)))$$
 immediately follows from the fact that the Finsler translation distance of $\gamma$ is attained at the projection $p_\gamma(\phi(x))$ of any point $\phi(x)$ in the image of the boundary map (Lemma \ref{lem:ellF}) together with the additivity of the Finsler distance along causal paths (Lemma \ref{orsum}). 
 Obviously the same inequality holds also if some peripheral element is not Shilov hyperbolic, as long as $\rho(\gamma)$ is, restricting the sum to the orthogeodesic corresponding to pairs of peripheral elements whose image is Shilov hyperbolic. 
\end{Remark}
\begin{Remark}
Note that we can rewrite Theorem \ref{thm:M0} as the Basmajian identity for cross-ratios associated to maximal representations, stated in the introduction as Theorem \ref{thm:cr}. Indeed, it follows from Lemma \ref{R(a,x,y,b)} that
$$\dF(p_\g(\phi(\delta_\alpha^+)),p_\g(\phi(\delta_\alpha^-)))=\frac{1}{2}\log\B(\gamma^-,\delta_\alpha^-,\gamma^+,\delta_\alpha^+)$$
and we have already noticed in Section \ref{sec:distandmaxreps} that for any element $\g\in\Gamma$
$$\ell^F(\gamma)=\frac{1}{2}\ell_{\B}(\gamma).$$
\end{Remark}
To prove Theorem \ref{thm:M0}, we will use the same strategy explained in Section \ref{classicalmeasure0}. The similarities will be evident as we will consider the upper-half space model for the symmetric space associated to $\Sp(2,\R)$; the maximum eigenvalue of the imaginary part of a point in $\calX$ will play the role of the imaginary part of a point in $\hyp$.

Consider $[\![x,\g x]\!]=\{x,\g x\}\cup(\!(x,\g x)\!)\subset \partial\hyp$ and denote by $\ell$ the Finsler translation distance of $\g$. We can assume that $\phi(\g^+)=l_\infty, \phi(\g^-)=0, \phi(x)=\Id$, up to conjugating with an element in $\Sp(2n,\R)$.  We define a monotone map
\begin{align*}\theta:[\![x,\g x]\!]&\to[0,\ell]\\
y&\mapsto\dF(p_\g(\phi(y)),p_\g(\phi(x)))=\frac12\log\det(\phi(y)\phi(x)^{-1})
\end{align*}
where the equality follows from Lemma \ref{Finslerdistance} and Remark \ref{projcausal} (see also the proof of Lemma \ref{orsum}).
For any orthotube $\alpha$ so that $\delta_\alpha^+,\delta_\alpha^-\in (\!(x,\gamma\cdot x)\!)$, 
let $I_\alpha$ be
$$I_\alpha:=(\theta(\delta_\alpha^+),\theta(\delta_\alpha^-))\subset [0,\ell],$$
and $X=\bigcup_{\alpha}I_\alpha$.

Proving Theorem \ref{thm:M0} is then equivalent to showing that $\lm(X)=\ell$, which in turn is the same as proving that $\lm([0,\ell]\setminus X)=0$. We want to use Lemma \ref{criterionmeasure0}; as in the classical case, we know that we can reduce ourselves to consider subintervals of $[0,\ell]$ with endpoints in $\theta(\Lambda(\Gamma))$. So it is enough to prove the following:
\begin{prop}\label{subinterval}
There exists a constant $\lambda>0$  such that for every $a,b\in\theta(\L(\G))$ there is an interval $I_\alpha\subset [a,b]\cap X$ with $\lm(I_\alpha)\geq \l(b-a)$.
\end{prop}

To be able to use Remark \ref{distboundary} in this setting, we first need a $\rho$-equivariant map from the unit tangent bundle of the hyperbolic plane into $\calX$. Recall that we fixed a cocompact action $h$ of $D\G$ on $\H^2$. We parametrize the unit tangent bundle by positively oriented triples of points $(a,b,c)$ on the boundary of the hyperbolic plane: a point $(p,v)\in T^1\hyp$ determines an oriented geodesic $l$ and $a$ and $c$ denote its start and end points. Moreover, $b$ is the point at infinity of the geodesic ray starting from $p$, orthogonal to $l$ and to the right of $l$.
\begin{figure}[H]
\begin{overpic}{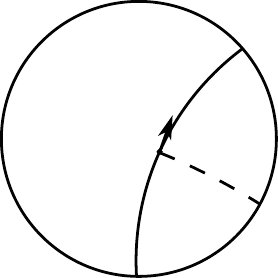}
\put(45,-10){$a$}
\put(100,20){$b$}
\put(90,85){$c$}
\put(25,45){$(p,v)$}
\end{overpic}
\caption{The parametrization of $T^1\hyp$}
\end{figure}

We use the boundary map $\phi$ to define a $D\rho$-equivariant map
\begin{align*}
F:T^1\hyp&\to \calX\\
(a,b,c)&\mapsto p_{\Yy_{\phi(a),\phi(c)}}(\phi(b)).
\end{align*}
\begin{prop}\label{qiembedding}
The map $F$ is a quasi-isometric embedding with respect to the Riemannian and Finsler metric on $\calX$.
\end{prop}
\begin{proof}
Since the Riemannian and Finsler metrics on $\calX$ are quasi-isometric, it is enough to prove the result for the Riemannian metric.
It is shown in \cite[Corollary 6.2]{bilw} that the restriction of $F$ to a $D\G$-orbit is a quasi-isometric embedding. The result then follows since $D\G$ acts cocompactly on $\H^2$ and $F$ is continuous (since $\phi$ is).
\end{proof}

In the proof of Proposition \ref{subinterval} it will be useful to be able to relate the maximum eigenvalue of $B-A$ to the logarithm of $\det BA^{-1}$, for $A,B\in \phi(\partial \G)$. 
\begin{lem}\label{lem:topvsdet}
Assume $(\Id, A, B, m\Id)$ is maximal, for some $m>1$. Then there exists $k_1,k_2$ depending only on $m$ and $n$ such that
$$k_2\log\det(BA^{-1})>\max\ev(B-A)> k_1\log\det(BA^{-1}).$$ 
\end{lem}
\begin{proof}
Since $(\Id, A, B, m\Id)$ is maximal, Lemma \ref{lem:max} and Lemma \ref{lem:eigenvalues} imply that all eigenvalues of $A$ and $B$ are between $1$ and $m$. Moreover it follows from  Lemma \ref{lem:eigenvalues}  that all eigenvalues of $BA^{-1}$ are bigger than 1, and since $B-B^{1/2} A^{-1}B^{1/2}B$ is positive definite they are also smaller than $m$ (compare with Lemma \ref{lem:evAB}). Moreover $B-A$ is conjugate to $(A^{-1/2}BA^{-1/2}-\Id)A$, so by Lemma \ref{lem:evAB} (and since the eigenvalues of $A$ are bigger than $1$) we have
$$\max\ev(BA^{-1})-1\leq\max \ev (B-A)<\max\ev (A)\left(\max \ev(BA^{-1})-1\right).$$
Because $\log x\leq x-1$ for any positive $x$,
$$\left(\max\ev(BA^{-1})-1\right)\geq \log\max\ev(BA^{-1})\geq \frac{1}{n}\log\det(BA^{-1})$$
where in the second inequality we used the fact that $(\max\ev(BA^{-1}))^n\geq \det(BA^{-1})$. 

Furthermore, $\frac{\log x}{x-1}$ is monotone decreasing for $x>1$, so
$$\max \ev(BA^{-1})-1\leq \frac{m-1}{\log(m)}\log \max \ev(BA^{-1})\leq\frac{m-1}{\log(m)}\log\det(BA^{-1}).$$
Combining these inequalities and using the fact that the eigenvalues of $A$ are smaller than $m$ we get the result.
\end{proof}
In the hyperbolic setting, given two points $x,y\in\Lambda$, we know that there exists a point $p\in\wt{S}\subset\H^2$ with $\im p=\frac{y-x}{2}$ and belonging to the geodesic of endpoints $x$ and $y$. In the case of maximal representations, given $x,y\in \Lambda$, with $\phi(x)=A$ and $\phi(y)=B$, we cannot guarantee that there is a point with imaginary part $\frac{1}{2}(B-A)$ belonging to the path $p_{A,B}(\phi(z))$, for $z\in(\!(x,y)\!)$. Our next goal is to show that there exists  $z\in (\!(x,y)\!)$ so that the maximum eigenvalue of the imaginary part of $p_{A,B}(\phi(z))$ is large enough with respect to the maximum eigenvalue of $B-A$. We begin by computing the imaginary part of the projection onto a tube with endpoints $(-C^2,C^2)$:
\begin{lem}\label{lem:0}
Assume that $(-C^2,T,C^2)$ is a maximal triple. Then
$$\im p_{-C^2,C^2}(T)=C(\Id-M^2)(\Id+M^2)^{-1}C,$$
where $M=C^{-1}TC^{-1}.$
\end{lem}
\begin{proof}
Let us choose $\g\in\Sp(2n,\R)$ such that $\g(0,l_\infty)=(-C^2,C^2)$. For example 
$$\g=\left(\begin{array}{cc}
\frac{1}{\sqrt{2}}C &-\frac{1}{\sqrt{2}}C\\
\frac{1}{\sqrt{2}}C^{-1}& \frac{1}{\sqrt{2}}C^{-1}
\end{array}\right).$$
The preimage $Z:=\g^{-1} \cdot T\in (\!(0,l_\infty)\!)$ is
$$Z=(C^{-1}T+C)(-C^{-1}T+C)=(C^{-1}TC^{-1}+\Id)(-C^{-1}TC^{-1}+\Id)^{-1}.$$
Hence
\begin{align*}
p_{-C^2,C^2}(T)&=\g( p_{\Yy_{0,\infty}}(\g^{-1} \cdot T))=\g\cdot i Z=\\
&=(iCZ-C)(iC^{-1}Z+C^{-1})^{-1}=C(iZ-\Id)(iZ+\Id)^{-1}C=\\
&=C(iZ-\Id)(iZ-\Id)(iZ-\Id)^{-1}(iZ+\Id)^{-1}C
\end{align*}
so
$$\im p_{-C^2,C^2}(T)=2CZ(Z^2+\Id)^{-1}C,$$
which becomes, after some computations,
$$\im p_{-C^2,C^2}(T)=C(\Id-M^2)(\Id+M^2)^{-1}C$$
where $M=C^{-1}TC^{-1}$ as required.
\end{proof}
The following Proposition is the hardest step in the proof of Theorem \ref{thm:M0}, and requires a careful analysis of causal paths.
\begin{prop}\label{prop:high}
There is an uniform constant $k_3$ depending on the representation $\rho$ such that, for any $(\!(x,y)\!)\subset (\!(\g^-,\g^+)\!)$, there exists $z\in (\!(x,y)\!)$ with 
$$\max\ev({\rm Im} p_{\phi(x),\phi(y)}(\phi(z)))>k_3\max\ev(\phi(y)-\phi(x)).$$
\end{prop}
Note that while the constant $k_3$ is independent on $\gamma$, for the statement to make sense we need to know what $\phi(x)$ and $\phi(y)$ are in $(\!(0,l_\infty)\!)=(\!(\phi(\gamma^-),\phi(\gamma^+))\!)$, so that we can identify them with positive definite symmetric matrices.
\begin{proof}
Up to the action of an element $H\in\Sp(2n, \R)$ of the form $H=\bsm\Id& K\\0&\Id\esm$ (whose action doesn't change the imaginary part) we can assume $\phi(x)=-C^2$ and $\phi(y)=C^2$, for some positive definite matrix $C$. So
$$2\max\ev(C^2)=\max \ev(\phi(y)-\phi(x)).$$ 
As in Lemma \ref{lem:0} we will denote 
$$M(t):=C^{-1}\phi(t)C^{-1}$$
for $t\in(\!(x,y)\!)$. $M(t)$ forms a continuous monotone curve with $M(x)=-\Id$ and $M(y)=\Id$.

Without loss of generality, we can assume that $e_1$ is an eigenvector of $C$ for its maximum eigenvalue. Then we get
\begin{align*}
\max\ev\left({\rm Im} p_{\phi(x),\phi(y)}(\phi(t))\right)&=\max_{\|v\|=1}\langle C(\Id-M(t)^2)(\Id+M(t)^2)^{-1}Cv,v\rangle\\
&\geq \langle C(\Id-M(t)^2)(\Id+M(t)^2)^{-1}Ce_1,e_1\rangle\\
&=\max \ev (C^2)[(\Id-M(t)^2)(\Id+M(t)^2)^{-1}]_{11}.
\end{align*}
So it is enough to show that there exists an uniform $k_3>0$ such that for each such path there exists $t\in(\!(x,y)\!)$ with $[(\Id-M(t)^2)(\Id+M(t)^2)^{-1}]_{11}>2k_3$.

Set $X(t)=M(t)^2$. Notice that, for every $t\in(x,y)$, the matrix $X(t)$ is a positive semidefinite matrix, and all its eigenvalues and diagonal coefficients are smaller than one: indeed, by monotonicity, all eigenvalues of $M(t)$ are in absolute value smaller than one.

Since any  matrix in $\Sym(n-1,\R)$ is orthogonally congruent to a diagonal matrix, for every $t$ there exist a orthogonal matrix $D(t)$ of the form $\bsm1&0\\0&D_1(t)\esm$ with $D_1(t)\in {\O}(n-1)$  such that $D(t)XD(t)^{-1}$ has the form 
\begin{equation}\label{eqn:6.11}\bpm a(t)&b_2(t)&\ldots&b_n(t)\\
b_2(t)&c_2(t)&\ldots&0\\
\vdots&&\ddots&0\\
b_n(t)&0&\ldots&c_n(t)\epm.\end{equation} 
Here the only non-zero values of $D(t)XD(t)^{-1}$ are the numbers $a(t),b_i(t),c_i(t)$, and $c_i(t)$ are the eigenvalues of the $(n-1)$-dimensional lower right block of $X(t)$. Observe that $a(t)=\left(X(t)e_1,e_1\right)$ and there exist vectors $v_i(t)\in\R^n$ of norm one such that  $c_i(t)=\left(X(t)v_i,v_i\right)$. In particular, since, for all $t$, $X(t)$  is a positive semidefinite matrix whose eigenvalues are strictly smaller than one, we deduce $0\leq c_i(t)<1$ and $0\leq a(t)< 1$.

We get from Equation \ref{eqn:6.11}
$$x(t):=[(\Id+X(t))^{-1}]_{11}=\frac{\det (\Id+X(t)|_{\<e_2,\ldots,e_n\>})}{\det (\Id+ X(t))}=\left(1+a(t)-\sum_{i=2}^n\frac{b_i(t)^2}{1+c_i(t)}\right)^{-1},$$
where the third equality follows by expanding the determinant of $\Id+X(t)$ with respect to the first row. Here we denote by $(\Id+X(t)|_{\<e_2,\ldots,e_n\>}$ the $(n-1)$-dimensional lower right block of $X(t)$, the choice of the notation is motivated by the fact that we understand the symmetric matrix $\Id+X(t)$ as representing a bilinear form. In this notation we have $\det (\Id+X(t)|_{\<e_2,\ldots,e_n\>})=c_2(t)\ldots c_n(t)$

Now, assume by contradiction that for all $t$, $x(t)$ is smaller than $\frac12 +\rho$ with $\rho<\frac{1}{64n^3}$. Then
$$
1>a(t)=x(t)^{-1}-1+\sum_{i=2}^n\frac{b_i(t)^2}{1+c_i(t)}\geq \frac{1}{\frac{1}{2}+\rho}-1>1-4\rho.$$
Since $0\leq c_i(t)< 1$:
\begin{align*}
\|b(t)\|^2&=\sum_{i=2}^n{b_i(t)^2}\leq 2 \sum_{i=2}^n\frac{b_i(t)^2}{1+c_i(t)}\leq\\ &
\leq 2(1+a(t)-x(t)^{-1})<2\left(2-\frac{1}{\frac{1}{2}+\rho}\right)<8\rho.
\end{align*}
Let us now write the matrix $M(t)$ in block form as 
$$M(t)=\bpm m(t) &^td(t)\\d(t)&N(t)\epm$$
where $N(t)$ is a matrix in $\Sym(n-1,\R)$.
We have 
$$X(t)=\bpm m(t)^2+\|d(t)\|^2&m(t)^td(t)+{^td(t)}N(t)\\N(t)d(t)+m(t)d(t)&d(t)^td(t)+N(t)^2\epm. $$
The rest of the proof is devoted to deducing a contradiction from the fact that $N(t)$ is a monotone path in $\Sym(n-1,\R)$, $-m(t)$ is a strictly decreasing function, but $N(t)$ has almost eigenvectors for the value $m(t)$ (since $\|b(t)\|=\|N(t)d(t)+m(t)d(t)\|$ is very small).

In order to make this idea precise, we focus on the subinterval $J$ of $(\!(x,y)\!)$ on which $m(t)^2$ is smaller than $1/2-4\rho$ (which is nonempty because $m(t)$ varies continuously between $-1$ and $1$). Then
$$1-4\rho <a(t)=m(t)^2+\|d(t)\|^2< \frac{1}{2}-4\rho+\|d(t)\|^2,$$
which implies that $\|d(t)\|^2>\frac{1}{2}$.

Fix an orthonormal basis $v_2(t),\dots, v_n(t)$ of eigenvectors for $N(t)$ (these exist since $N(t)$ is symmetric), where $v_i(t)$ is an eigenvector corresponding to the eigenvalue $n_i(t)$. Observe that since $M(t)$ forms a monotone path, also $N(t)$ forms a monotone path in $\Sym(n,\R)$, and in particular $n_i(t)$ are monotone functions with $n_i(x)=-1$, $n_i(y)=1$.
We can write
$$d(t)=\sum_{i=2}^n \alpha_i(t)v_i(t).$$
Then
$$N(t)d(t)+m(t)d(t)=\sum_{i=2}^n\alpha_i(t)(n_i(t)+m(t))v_i(t)$$
which implies, taking the norm squared, that:
$$\sum_{i=2}^n\alpha_i(t)^2(n_i(t)+m(t))^2<8\rho$$
and hence for each $i$
$$|\alpha_i(t)(n_i(t)+m(t))|<\sqrt{8\rho}.$$
Note that since $\sum_{i=2}^n\alpha_i(t)^2=\|d(t)\|^2 >\frac{1}{2}$, for every $t$ there is an $i$ such that $\alpha_i(t)>\frac{1}{\sqrt{2n}}$, i.e.\ the sets $J_i=\left\{t\in J | \alpha_i(t)>\frac{1}{\sqrt{2n}}\right\}$ form an open cover of $J$. But in each $J_i$
$$|n_i(t)+m(t)|<4\sqrt{n\rho}$$
that is, $n_i(t)$ and $-m(t)$ are very close. Since $n_i(t)$ is strictly increasing and $-m(t)$ is strictly decreasing, the $J_i$ are connected.

Consider the subinterval $I$ of $J$ where $m(t)^2\leq \frac{1}{4}$. We can write it as 
$$I=[x_0,x_1]\cup [x_1,x_2]\cup\dots\cup[x_{k-1},x_k]$$
where $k\leq n$ and for each $j$ there exists an $i$ such that $[x_{j-1},x_j]\subset J_i$.
We know that $m(x_k)-m(x_0)=1$; at the same time $$m(x_k)-m(x_0)=\sum_{j=1}^{k}m(x_j)-m(x_{j-1}).$$
Since
\begin{align*}
m(x_j)-m(x_{j-1})&=m(x_j)+c_i(x_j)-c_i(x_j)-m(x_{j-1})+c_i(x_{j-1})-c_i(x_{j-1})=\\
&\leq 8\sqrt{n\rho}+\underbrace{(c_i(x_{j-1})-c_i(x_j))}_{< 0}<8\sqrt{n\rho}
\end{align*}
we have $1=m(x_k)-m(x_0)\leq 8n\sqrt{n\rho}$
which gives a contradiction because $\rho$ is smaller than $\frac{1}{64n^3}$.
\end{proof}
The following Lemma should be understood as an analogue, in the Siegel space, of the fact that the closest point to $i$ in the horoball $\{x+iy_0|x\in\R\}\subset \hyp$ is $iy_0$.
\begin{lem}\label{lem:closetohigh}
For every $Z\in\Sym(n,\R)$ and $Y,W\in \Sym^+(n,\R)$ with $Y-W$ positive definite, we have:
$$\dR(iY,Z+iW)\geq \log \frac{\max\ev(Y)}{\max \ev(W)}.$$
\end{lem}
\begin{proof}
We know that
$$\dC(iY,Z+iW)_j=\log\left(\frac{1+\sqrt r_j}{1-\sqrt r_j}\right)$$
where $1\geq r_1\geq\ldots\geq r_n\geq 0$ are the eigenvalues of the cross-ratio $R(iY, Z-iW,Z+iW,-iY)$.
Since the function 
$$x\mapsto \log\left(\frac{1+x}{1-x}\right)$$
is monotone increasing on $(0,1)$ and $\dR(iY,Z+iW)\geq\dC(iY,Z+iW)_1$, we want to estimate the quantity $\log\left(\frac{1+\sqrt r_1}{1-\sqrt r_1}\right)$.

We set $S=Y^{-1/2}ZY^{-1/2}$ and $T=Y^{-1/2}WY^{-1/2}$. By Lemma \ref{lem:eigenvalues}, $T$ is positive definite and all its eigenvalues are smaller than 1. One can show by explicit computations that
$R(iY, Z-iW,Z+iW,-iY)$ is conjugate to $$\Id -4[S^2+(T+\Id)^2+i(ST-TS)]^{-1}T=\Id-X^{-1}$$
where $X=\frac{1}{4}T^{-1/2}[S^2+(T+\Id)^2+i(ST-TS)]T^{-1/2}$. Now $r_1$ is $1-x_1^{-1}$, where $x_1$ is the maximum eigenvalue of $X$.

Note that $X$ is Hermitian, so its maximum eigenvalue is at least the maximum eigenvalue of its real part: indeed, because $\im(X)$ is antisymmetric, if $\lambda$ is the maximum eigenvalue of $\re(X)$ and $v$ is an eigenvector for $\lambda$ of norm one, we have
$$x_1\geq \left((\re(X)+i\im(X))v,v\right)=\lambda+i\left(\im(X)v,v\right)=\lambda.$$
As a consequence,
$$x_1\geq \max\ev(\re(X))\geq \max\ev\left(\frac{1}{4}T^{-1}(T+\Id)^2\right)=\frac{1}{4}\frac{(m+1)^2}{m},$$
where $m$ is the minimum eigenvalue of $T$. Thus
$$r_1=1-\frac{1}{x_1}\geq \frac{(m-1)^2}{(m+1)^2}$$
which implies that
$$\sqrt{r_1}\geq \frac{1-m}{m+1},$$
because  $m<1$. Using Lemma \ref{lem:evAB}, we get
$$\frac{1+\sqrt r_1}{1-\sqrt r_1}\geq \frac{1}{\min\ev(T)}=\max\ev(T^{-1})\geq \max\ev (Y)\min\ev(W^{-1})=\frac{\max\ev (Y)}{\max\ev(W)}.$$
\end{proof}

Next we prove that if $Y\in \Yy_{C,D}$ has imaginary part with large maximum eigenvalue, the segment with endpoints $\theta(C)$ and $\theta(D)$ in $[0,\ell]$ is long.
\begin{lem}\label{lem:highislarge}
If $Y\in \Yy_{C,D}$, then $\frac{1}{2}\max \ev (D-C)\geq \max \ev ({\rm Im} Y)$.
\end{lem}
\begin{proof} Up to translating horizontally (as at the beginning of the proof of Proposition \ref{prop:high}), we can assume $C=-E^2$ and $D=E^2$, for some matrix $E$. We know from Lemma \ref{lem:0} that the imaginary part of a point $Y\in\Yy_{-E^2,E^2}$ can be written as $E(\Id-M^2)(\Id+M^2)^{-1}E$ for some positive semidefinite matrix $M$ whose eigenvalues are smaller than 1. By Lemma \ref{lem:evAB} we have
$$
\max \ev(\im Y)\leq (\max \ev (E))^2\max \ev (\Id-M^2)(\Id+M^2)^{-1}\leq  \max \ev (E^2).
$$
\end{proof}
We can now prove the key step for Theorem \ref{thm:M0}.
\begin{proof}[Proof of Proposition \ref{subinterval}]Given $a=\theta(s)$ and $b=\theta(t)$, with $s,t\in\L(\G)$, we want to find a boundary component $\delta_\alpha$ with $d-c:=d^F (p_\gamma(\phi(\delta_\alpha^-)),p_\gamma(\phi(\delta_\alpha^+)))\geq \l (b-a)$, for some $\lambda$ independent on $a$ and $b$.
Let $z\in (\!(s,t)\!)$ given by Proposition \ref{prop:high}. By Remark \ref{distboundary} and Proposition \ref{qiembedding}, we know that there exist $k_4>1$ and 
$w\in F\left( T^1{\partial\wt\Sigma}\right)$ 
at distance at most $\log k_4$ from $p_{\phi(s),\phi(t)}(\phi(z))$. Here $F:T^1\mathbb H^2\to \mathcal X$ is the $D\rho$-equivariant quasi-isometric embedding studied in Proposition \ref{qiembedding}. Denote by $s'<t'$ the points in $(\!(s,t)\!)$ such that $w\in\Yy_{\phi(s'),\phi(t')}$ and let $c=\theta(s')$ and $d=\theta(t')$.

Note that $(\phi(x),\phi(s'),\phi(t'),\phi(\gamma\cdot x))$ is maximal, and hence so is $(\phi(x),\phi(s'),\phi(t'),m\Id)$ where $m=\max\ev(\phi(\gamma\cdot x))$. But since the translation length of $\gamma$ is $\ell$ and we assumed that $\phi(x)=\Id$ (which implies that all eigenvalues of $\phi(\gamma\cdot x)$ are bigger than one), we can deduce that $m<2e^\ell$ and hence
$$(\phi(x),\phi(s'),\phi(t'),2e^\ell\Id)$$
is maximal.

We have $[c,d]\subset [a,b]\cap X$ and
\begin{align*}
d-c&=\frac{1}{2}\log\det(\phi(t')\phi(s')^{-1})\\
&\stackrel{(i)}{>}\frac{1}{2k_2}\max \ev (\phi(t')-\phi(s'))\\
&\stackrel{(ii)}{\geq} \frac{1}{k_2}\max \ev(\im(w))\\
&\stackrel{(iii)}{\geq} \frac{1}{k_2k_4} \max \ev(\im p_{\phi(s),\phi(t)}(\phi(z)))\\
&\stackrel{(iv)}{>}\frac{k_3}{k_2k_4} \max \ev (\phi(t)-\phi(s))\\
&\stackrel{(v)}{>} \frac{2k_1k_3}{k_2k_4}(b-a)
\end{align*}
where we apply Lemma \ref{lem:topvsdet} in (i) and (v), Lemma \ref{lem:highislarge} in (ii), Lemma \ref{lem:closetohigh} in (iii) and Proposition \ref{prop:high} in (iv).\end{proof}
\begin{figure}[H]
\begin{overpic}{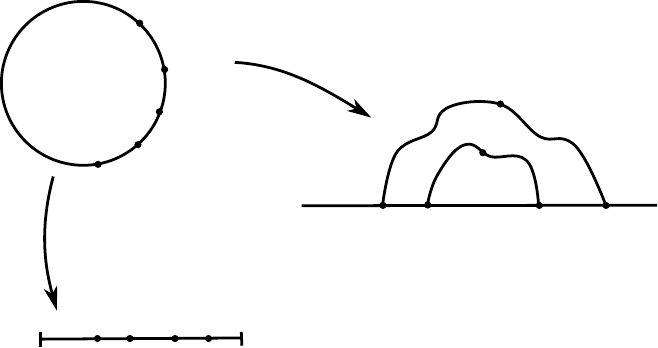}
\put(-5,50){$\partial\hyp$}
\put(15,24){$s$}
\put(20,27){$s'$}
\put(27,35){$z$}
\put(27,42){$t'$}
\put(23,50){$t$}
\put(45,42){$\phi$}
\put(75,40){$p_{\phi(s),\phi(t)}(\phi(z))$}
\put(70,25){$w$}
\put(49,16){$\phi(s)$}
\put(61,16){$\phi(s')$}
\put(75,16){$\phi(t')$}
\put(90,16){$\phi(t)$}
\put(3,15){$\theta$}
\put(5,-5){$0$}
\put(13,-5){$a$}
\put(18,-5){$c$}
\put(25,-5){$d$}
\put(31,-5){$b$}
\put(38,-5){$\ell$}\end{overpic}
\caption{A schematic picture of the situation described in the proof}
\end{figure}
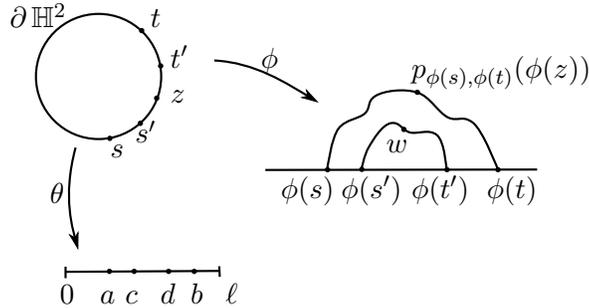
\section{Geometric Basmajian-type inequalities}
This section is dedicated to the proof of Theorem \ref{thm:main} and to showing that the gap between the middle and right-hand side term in the inequalities \eqref{Finsler} and \eqref{Riemannian} can be arbitrarily large.

We will need a preliminary lemma concerning the vectorial length of orthotubes.
\begin{lem}\label{2logcothC}
Given two peripheral elements $\g,\delta$, we have
$$ \dC(p_\g(\phi(\delta^-)),p_\g(\phi(\delta^+)))_i= 2\log\coth \frac{\ellC(\ort(\gamma,\delta))_{n-i+1}}{2}.$$
\end{lem}
\begin{proof} Let $\Yy_{a,b}$ be the $\R$-tube orthogonal to both $\Yy_\g$ and $\Yy_\delta$.
Since $(\phi(\g^-),\phi(\delta^+),\phi(\delta^-),\phi(\g^+))$ is a maximal $4$-tuple, by Lemma \ref{lem:IdLambda} we can assume (up to the action of $\Sp(2n,\R)$) that $$(\phi(\g^-),\phi(\delta^+),\phi(\delta^-),\phi(\g^+))=(-\Id,-\Lambda,\Lambda,\Id),$$ where $\Lambda=\diag\left(\frac{1}{\lambda_n},\dots,\frac{1}{\l_1}\right)$, with $\l_1\geq\ldots\geq \l_n>1$. 
\vspace{0.5cm}
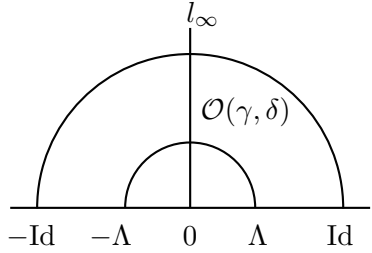
\begin{figure}[H]
\begin{overpic}{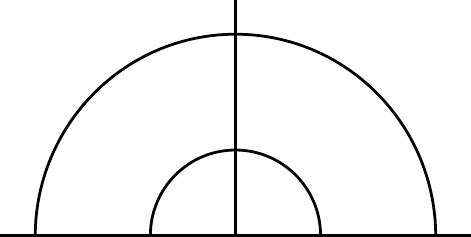}
\put(-1,-10){$-\Id$}
\put(22,-10){$-\L$}
\put(48,-10){$0$}
\put(66,-10){$\L$}
\put(88,-10){$\Id$}
\put(53,25){$\ort(\gamma,\delta)$}
\put(49,52){$l_\infty$}
\end{overpic}
\vspace{0.5cm}
\caption{A schematic picture of two $\R$-tubes and their orthogeodesic}
\end{figure}

By Lemma \ref{perpto0infty}, $\Yy_{0,l_\infty}$ is the $\R$-tube orthogonal to $\Yy_\g$ and $\Yy_\delta$ and it intersects them in $i\Id$ and $i\Lambda$. As a consequence, by Lemma \ref{Finslerdistance}
$$\ellC(\ort(\gamma,\delta))=(\log\l_1,\ldots,\log\l_n).$$
Moreover, Lemma \ref{R(a,x,y,b)} implies that the distance $\dC(p_\g(\phi(\delta^-)),p_\g(\phi(\delta^+)))$ is the vector given by the logarithm of the eigenvalues of $R(\phi(\g^-),\phi(\delta^+),\phi(\delta^-),\phi(\g^+))$. Now using \eqref{cross-ratio} we obtain
\begin{align*}
&R(\phi(\g^-),\phi(\delta^+),\phi(\delta^-),\phi(\g^+))=
(\Id-\Lambda)^{-2}(\Lambda+\Id)^2,
\end{align*}
so
$$\dC(p_\g(\phi(\delta^-)),p_\g(\phi(\delta^+)))=\left(\log\left(\frac{\l_n+1}{\l_n-1}\right)^2,\dots,\log\left(\frac{\l_1+1}{\l_1-1}\right)^2\right).$$
\end{proof}

\subsection{The Finsler metric}
We show the Basmajian-type inequality for the Finsler metric and for a single boundary component:
\begin{theor}\label{Finsler-component}
Let $\Sigma$ be a compact surface with boundary and $\gamma\in\Gamma=\pi_1(\Sigma)$ a peripheral element. Given an Anosov maximal representation $\rho:\Gamma\to\Sp(2n,\R)$, we have
$$n\!\!\sum_{\alpha\in\ort_\Sigma(\gamma)}\log \coth \frac{\ellC(\alpha)_n}{2}\geq\ellF(\rho(\gamma))\geq n\!\! \sum_{\alpha\in\ort_\Sigma(\gamma)}\log\coth\frac{\ellF(\alpha)}{n}$$
with equality if and only if $\rho$ is the diagonal embedding of a hyperbolization.
\end{theor}
This result, together with Remark \ref{sumorthogeodesics}, implies the Finsler metric part of Theorem \ref{thm:main}.

The first step is to relate $\ellF(\ort(\gamma,\delta))$ to the distance between the orthogonal projections of the extremal points of $\Yy_\d$ on $\Yy_\gamma$. For the reader's convenience, since we will use multiple properties of $\log\coth(x)$, Figure \ref{plotlogcoth} shows the graph of this function.
\begin{figure}[h]
\includegraphics[width=.5\textwidth]{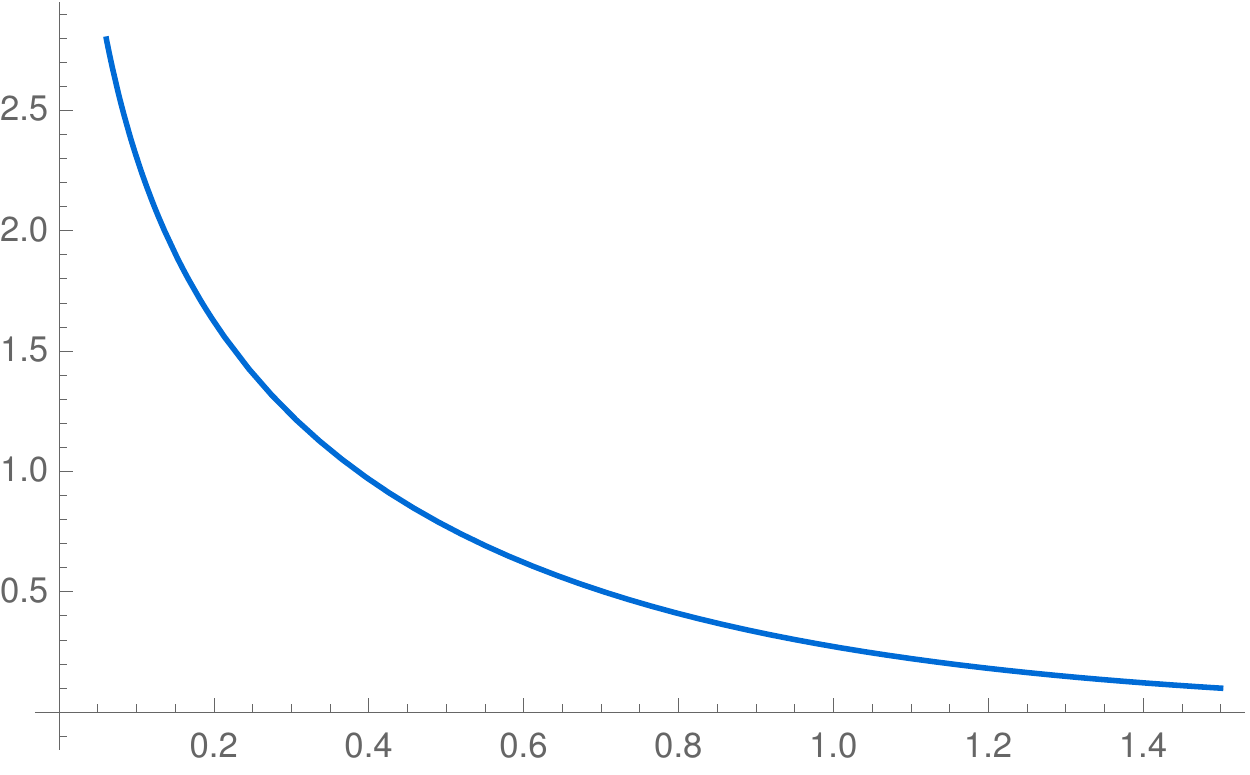}
\caption{The graph of $\log\coth(x)$ (drawn using Mathematica)}\label{plotlogcoth} 
\end{figure}

\begin{lem}\label{2logcoth}
Given two peripheral elements $\g,\delta$, we have
$$n\log\coth\frac{\ellC(\ort(\g,\delta))_n}{2}\geq \dF(p_\g(\phi(\delta^-)),p_\g(\phi(\delta^+)))\geq n\log\coth 
\frac{\ellF(\ort(\gamma,\delta))}{n},$$
with equalities if and only if all the eigenvalues of $R(\phi(\g^-),\phi(\delta^+),\phi(\delta^-),\phi(\g^+))$ are equal.
\end{lem}
\begin{proof}
We have
\begin{align*}\dF(p_\g(\phi(\delta^-)),p_\g(\phi(\delta^+)))&=\sum_{i=1}^n\log\coth\frac{\ellC(\ort(\g,\delta))_i}{2} \\
&\geq n\log\coth\frac{\sum_{i=1}^n\ellC(\ort(\g,\delta))_i}{2n}\\
&=n\log\coth\frac{\ellF(\ort(\g,\delta)))}{n}
\end{align*}
where the first equality is given by Lemma \ref{2logcothC} and the inequality follows from the convexity of $\log\coth(x)$.
On the other hand by the monotonicity of $\log\coth(x)$
$$\dF(p_\g(\phi(\delta^-)),p_\g(\phi(\delta^+)))=\sum_{i=1}^n\log\coth\frac{\ellC(\ort(\g,\delta))_i}{2} \leq n\log\coth\frac{
\ellC(\ort(\g,\delta))_n}{2}.$$

Since $\log\coth$ is strictly convex and strictly monotone, there are equalities if and only if $\ellC(\ort(\g,\delta))_i=\ellC(\ort(\g,\delta))_j$ for all $i,j$. Because $\log\coth(x)$ is strictly decreasing, this is also equivalent to the fact that, for each pair $i,j$, $$\dC(p_\g(\phi(\delta^-)),p_\g(\phi(\delta^+))_i=\dC(p_\g(\phi(\delta^-)),p_\g(\phi(\delta^+))_j,$$
namely that  all eigenvalues of $R(\phi(\g^-),\phi(\delta^+),\phi(\delta^-),\phi(\g^+))$ are equal (Lemma \ref{R(a,x,y,b)}).
\end{proof}

In order to prove that the equalities characterize diagonal representations we need the following lemma:
\begin{lem}\label{lem:iff} 
Assume that an equality in Theorem \ref{Finsler-component} holds. Then $\pi^{\SL}\circ p_\g \circ\phi$ is constant on $\Lambda(\G)\setminus\{\g^-,\g^+\}$.
\end{lem}
\begin{proof} Denote by $\Phi_{\g,\SL}$ the map $\pi^{\SL}\circ p_\g \circ\phi$ and by $\Phi_{\g,\R}$ the map $\pi^{\R}\circ p_\g \circ\phi$.
We will prove that for each $x,y\in \Lambda(\G)\setminus\{\g^-,\g^+\}$ we have $\Phi_{\g,\SL}(x)=\Phi_{\g,\SL}(y)$.

Observe that since, by assumption, an equality holds in Theorem \ref{Finsler-component}, we have necessarily that the corresponding equality in Lemma \ref{2logcoth} is satisfied.
 If in particular, $\delta\in \G$ is such that $(x,\d^+,\d^-,y)$ is positively oriented, as a consequence of Lemma \ref{2logcoth} we have that all eigenvalues of $R(\phi(\gamma^-),\phi(\delta^+),\phi(\delta^-),\phi(\gamma^+))$ are equal. Since $R(\phi(\gamma^-),\phi(\delta^+),\phi(\delta^-),\phi(\gamma^+))$ is conjugated to a symmetric matrix, this means
$$R(\phi(\gamma^-),\phi(\delta^+),\phi(\delta^-),\phi(\gamma^+))=\l \Id$$
for some $\l$. Up to the action of the symplectic group, we can assume $$(\phi(\g^-),p_\g\phi(\delta^+), p_\g\phi(\delta^-), \phi(\g^+))=(0,i\Id,iA,l_\infty)$$
for some $A\in\Sym^+(n,\R)$. Under these assumptions, by \cite[Lemma 4.3]{bp},
$$R(\phi(\gamma^-),\phi(\delta^+),\phi(\delta^-),\phi(\gamma^+))=A.$$
Thus $A=\l\Id$, which implies that $\Phi_{\g,\SL}(\delta^-)=\Phi_{\g,\SL}(\delta^+)=\Id$.

In particular 
\begin{align*}\dSL(\Phi_{\g,\SL}(x),\Phi_{\g,\SL}(y))&\leq \dSL(\Phi_{\g,\SL}(x),\Phi_{\g,\SL}(\delta^{ +}))+\dSL(\Phi_{\g,\SL}(\delta^{-}),\Phi_{\g,\SL}(y))\\
&\leq \sqrt{n-1}\left(\Phi_{\g,\R}(\delta^+)-\Phi_{\g,\R}(x)+\Phi_{\g,\R}(y)-\Phi_{\g,\R}(\delta^-)\right),
\end{align*}
where the last inequality is a consequence of Lemma \ref{lem:pos}. By repeating the argument, we obtain the same inequality for any finite sum of $\delta$. Taking the limit we get
\begin{align*}\dSL(\Phi_{\g,\SL}(x),\Phi_{\g,\SL}(y))\leq  \sqrt{n-1}\left(\Phi_{\g,\R}(y)-\Phi_{\g,\R}(x)-\mkern-25mu\sum_{\delta:(\!(\delta^+,\delta^-)\!)\subset(\!(x,y)\!)}\mkern-20mu\left(\Phi_{\g,\R}(\delta^-)-\Phi_{\g,\R}(\delta^+)\right)\right)
\end{align*}
Using Remark \ref{dFandpiR} and Theorem \ref{thm:M0}, we deduce
\begin{align*}
\Phi_{\g,\R}(y)-\Phi_{\g,\R}(x)&=\frac{2}{\sqrt{n}}\dF(p_\g\phi(x),p_\g\phi(y))
\\&=\frac{2}{\sqrt{n}}  \sum_{\delta:(\!(\delta^+,\delta^-)\!)\subset(\!(x,y)\!)}\mkern-20mu \dF(p_\g\phi(\delta^-),p_\g\phi(\delta^+))
\\ &=\mkern-20mu\sum_{\delta:(\!(\delta^+,\delta^-)\!)\subset(\!(x,y)\!)}\mkern-20mu\left(\Phi_{\g,\R}(\delta^-)-\Phi_{\g,\R}(\delta^+)\right),
\end{align*}
so $\dSL(\Phi_{\g,\SL}(x),\Phi_{\g,\SL}(y))=0$, which means that $\Phi_{\g,\SL}(x)=\Phi_{\g,\SL}(y)$.
\end{proof}

\begin{proof}[Proof of Theorem \ref{Finsler-component}]
Theorem \ref{thm:M0} and Proposition \ref{2logcoth} imply that both inequalities in Theorem \ref{Finsler-component} hold.

If $\rho$ is the diagonal embedding of a hyperbolization, then the inequalities are identities by the classical Basmajian's identity.
 Let us then prove the reverse implication. Observe that the boundary of a diagonal disk is given by
$$\partial\mbox{Diag}=\{l_\infty\}\cup\{\lambda \Id\;|\;\lambda\in \R\}\subset \Ll(\R^{2n}).$$ 
So as a consequence of Lemma \ref{lem:iff} we get that, up to conjugating the representation in $\Sp(2n,\R)$, the image $\phi(\Lambda(\Gamma))$ is contained in $\partial\mbox{Diag}$. This implies that $\rho(\Gamma)$ preserves $\partial\mbox{Diag}$ and hence
$$\rho(\Gamma)\subset\mbox{Stab}_{\Sp(2n,\R)}(\partial\mbox{Diag}).$$
It can be directly checked that 
$$\mbox{Stab}_{\Sp(2n,\R)}(\partial\mbox{Diag})=\left\{\left(\begin{matrix}
aX&bX\\cX&dX\end{matrix}\right)\left|\left(\begin{matrix}
a&b\\c&d
\end{matrix}\right)\in \SL(2,\R), X\in \O(n)\right.\right\}\simeq\SL(2,\R)\times \O(n),$$ which finishes the proof.
\end{proof}
\subsection{The Riemannian metric}\label{Sec:Riemannian}
We begin the section deducing from Theorem \ref{thm:M0} an inequality for the Riemannian distances.
\begin{prop}\label{prop:m0Riemannian}
Let $\rho:\pi_1(\Sigma)\to\Sp(2n,\R)$ be an Anosov maximal representation, for $\Sigma$ with nonempty boundary. For every peripheral element $\gamma$ we have:
$$\ellR(\g)\leq\sum_{\alpha\in\ort_\Sigma(\gamma)}\dR(p_\g(\phi(\delta_\alpha^-)),p_\g(\phi(\delta_\alpha^+))).$$
\end{prop}
\begin{proof}
Consider the peripheral elements $\delta$ such that $(\!(\delta^+,\delta^-)\!)\subset(\!(x,\gamma x)\!)$ for some fixed $x\in\Lambda(\Gamma)$. As they are countably many, we can index them as $\{\delta_i\}_{i\in \N}$. For every $k\in \N$, consider $\{\delta_1,\dots,\delta_k\}$. Up to reordering them, assume that $\delta_1^+,\delta_1^-,\dots,\delta_k^+,\delta_k^-$ are positively oriented. Then $\ellR(\g)\leq A_k+B_k$, where
$$A_k:=\sum_{i=1}^{k}\dR(p_\g(\phi(\delta_i^+)),p_\g(\phi(\delta_i^-)))$$
and
$$B_k:=\dR(x,p_\g(\phi(\delta_1^+)))+\sum_{i=1}^{k-1}\dR(p_\g(\phi(\delta_i^-)),p_\g(\phi(\delta_{i+1}^+)))+\dR(p_\g(\phi(\delta_k^-)),\g x).$$
By definition, as $k\to \infty$
$$A_k\to \!\!\!\!\!\!\!\!\!\!\!\sum_{\delta :(\delta^+,\delta^-)\subset (x,\g x)}\!\!\!\!\!\!\!\!\!\!\dR(p_\g(\phi(\delta^-)),p_\g(\phi(\delta^+)))=\sum_{\alpha\in\ort_\Sigma(\gamma)}\dR(p_\g(\phi(\delta_\alpha^-)),p_\g(\phi(\delta_\alpha^+))).$$
So it is enough to show that $B_k\to 0$ as $k\to\infty$. For every $k$, by Lemma \ref{lem:relRF} we have
$$B_k\leq 2 B^F_k$$
where 
$$B^F_k=\dF(x,p_\g(\phi(\delta_1^+)))+\sum_{i=1}^{k-1}\dF(p_\g(\phi(\delta_i^+)),p_\g(\phi(\delta_{i+1}^-)))+\dF(p_\g(\phi(\delta_k^-)),\g x).$$
By Theorem \ref{thm:M0} $B^F_k$ tends to zero as $k$ tends to infinity.


\end{proof}
As in the Finsler case, the inequalities \eqref{Riemannian} follow from a computation for each boundary component (Theorem \ref{Riemannian-component}) and Remark \ref{sumorthogeodesics}.
\begin{theor}\label{Riemannian-component}
For any Anosov maximal representation $\rho:\Gamma\to\Sp(2n,\R)$ and any peripheral element $\gamma\in \Gamma$ we have
$$2\sqrt{n} \!\!\sum_{\alpha\in\ort_\Sigma(\gamma)}\!\!\log\coth\frac{\ellC(\alpha)_n}{2}\geq\ellR(\gamma)\geq 2\sqrt n\!\!\sum_{\alpha\in\ort_\Sigma(\gamma)}\!\!\log \coth \frac{\ellR(\alpha)}{2\sqrt{n}}$$
with equalities if and only if $\rho$ is the diagonal embedding of a hyperbolization.
\end{theor}
\begin{proof}
We have
$$
\ellR(\g) \geq \frac{2\ellF(\g)}{\sqrt n} \geq 2\sqrt{n}\!\!\sum_{\alpha\in\ort_\Sigma(\g)}\!\!\log\coth\frac{\ellF(\alpha)}{n} \geq 2\sqrt{n}\!\!\sum_{\alpha\in\ort_\Sigma(\gamma)}\!\!\log\coth\frac{\ellR(\alpha)}{2\sqrt{n}},
$$
where the first and last inequalities follow from Lemma \ref{lem:relRF} and the middle one is given by Theorem \ref{Riemannian-component}. Note that the first inequality has also a geometric interpretation: since $\rho(\g)$ preserves the tube $\Yy_\g$, the Riemannian translation length of $\rho(\g)$, $\ellR(\g)$ is at most the translation length in the tube, that, in turn, is at most the translation length in the Euclidean factor $\R$.

For the other inequality, we have:
$$\ellR(\g)\leq \sum_{\alpha\in\ort_\Sigma(\gamma)}\dR(p_\g(\phi(\delta_\alpha^-)),p_\g(\phi(\delta_\alpha^+)))\leq 2\sqrt{n} \sum_{\alpha\in\ort_\Sigma(\gamma)}\log\coth\frac{\ellC(\alpha)_n}{2}$$
where the first inequality is Proposition \ref{prop:m0Riemannian} and the second follows from Lemma \ref{2logcothC} and the monotonicity of $\log\coth(x)$.

If $\rho$ is the diagonal embedding of a hyperbolization, then the equalities holds as a consequence of the classical Basmajian equality. On the other hand assume that some equality holds. Then the corresponding equality holds in Theorem \ref{Finsler-component} and hence $\rho$ is a diagonal embedding.
\end{proof}
\subsection{Extremal cases}\label{arbitrarilybad}
In this section we show that there are sequences of maximal representations in which the length of the boundary components stay bounded away from zero, but the  $\R$-tubes associated to any two peripheral elements are arbitrarily far apart.
\begin{prop}\label{prop:arbitrarilybad}
For any $n\geq 2$, $L>0$ and $\eta>0$, there is a maximal representation $\rho:\G_{0,n+1}\to\Sp(2n,\R)$ and a peripheral element $\g$ such that
$$
\ellF(\g)=\frac{nL}{2}\quad\text{ and }\quad \ellR(\g)=\sqrt{n}L$$
while
$$n\sum_{\alpha\in\ort_\Sigma(\gamma)}\log \coth \frac{\ellF(\alpha)}{n}< \eta\quad\text{ and }\quad 2\sqrt n\sum_{\alpha\in\ort_\Sigma(\gamma)}\log \coth \frac{\ellR(\alpha)}{2\sqrt{n}}< \eta. $$
\end{prop}
The following lemma from hyperbolic geometry will be useful to prove Proposition \ref{prop:arbitrarilybad}:
\begin{lem}\label{lem:hypgeom}
For every $L>0$, $\varepsilon>0$, and $n\in \N$, there exists a hyperbolic structure on a sphere with $n+1$ boundary components such that one boundary component $\gamma_0$ has length $L$ and there is an orthogeodesic $\alpha$ between $\gamma_0$ and another boundary component satisfying
$$2\log\coth\frac{\ell(\alpha)}{2}=L-\varepsilon.$$
\end{lem}
\begin{proof}
Explicitly, we want $\alpha$ of length
$$g(L,\varepsilon)=\log\frac{e^{(L-\varepsilon)/2}+1}{e^{(L-\varepsilon)/2}-1}.$$
It is enough to construct a pair of pants with one boundary component of length $L$ and the orthogonal between this boundary and another one of length $g(L,\varepsilon)$. Equivalently, we can construct a right-angled hexagon with one side of length $\frac{L}{2}$ and an adjacent one of length $g(L,\varepsilon)$. By properties of hyperbolic hexagons (see \cite[Section 2.4]{buser_book}), it is enough to find positive $x$ and $y$ satisfying
$$\cosh y=\sinh\frac{L}{2}\sinh x \cosh g(L,\varepsilon)-\cosh\frac{L}{2}\cosh x.$$

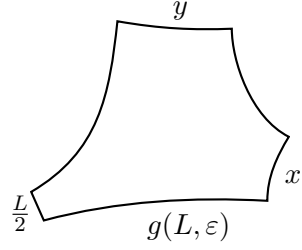
\begin{figure}
\vspace{.5cm}
\begin{overpic}{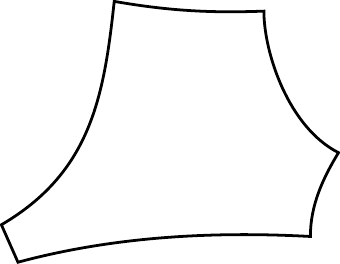}
\put(-8,0){$\frac{L}{2}$}
\put(45,-5){$g(L,\varepsilon)$}
\put(98,15){$x$}
\put(55,80){$y$}
\end{overpic}
\caption{The hexagon we want to construct}
\end{figure}

If we find such $x$ and $y$, there is a right-angled hyperbolic hexagon with non-consecutive sides of length $\frac{L}{2}$, $x$ and $y$ (and the side between the ones of length $\frac{L}{2}$ and $x$ has length $g(L,\varepsilon)$). But since $\cosh$ is bijective between $(0,\infty)$ and $(1,\infty)$, it is enough to show that there exists $x>0$ such that
$$\sinh\frac{L}{2}\sinh x \cosh g(L,\varepsilon)-\cosh\frac{L}{2}\cosh x>1. \;\; (\star)$$
Now there is some positive $\eta$ such that
$$\cosh g(L,\varepsilon)=\frac{e^{L-\varepsilon}+1}{e^{L-\varepsilon}-1}=\frac{e^L+1}{e^L-1}+\eta=\frac{\cosh(L/2)}{\sinh(L/2)}+\eta.$$
Using this, inequality $(\star)$ becomes
$$\eta \sinh x \sinh\frac{L}{2}>(\cosh x-\sinh x)\cosh\frac{L}{2}+1.$$
Let $x_0$ be such that $\cosh x-\sinh x<1$; then it is enough to choose $x>x_0$ satisfying
$$\sinh x>\frac{1+\cosh\frac{L}{2}}{\delta\sinh\frac{L}{2}}.$$
\end{proof}
\begin{proof}[Proof of Proposition \ref{prop:arbitrarilybad}]
Choose $0<\varepsilon=\frac{\eta}{n^2}$. 
We write the fundamental group $\G_{0,n+1}$ of an $(n+1)$-punctured sphere as 
$$\G_{0,n+1}=\left\<\g_0,\ldots,\g_n\left|\prod\g_i=1\right.\right\>.$$
Consider $n$ hyperbolizations $\rho_i:\G_{0,n+1}\to\SL(2,\R)$ so that $\rho_i(\g_0)$ has length $L$ and the orthogeodesic $\alpha_i$ between the axis of $\g_0$ and the axis of $\g_i$ satisfies the hypothesis of Lemma \ref{lem:hypgeom}. Consider the representation $\rho=\diag(\rho_1,\dots,\rho_n)$. Then
$$\ellF(\g_0)=\frac{nL}{2} \quad \text{ and }\quad\ellR(\g_0)=\sqrt{n}L.$$

The classical Basmajian's identity tells us that
$$2\sum_\alpha \log\cosh\frac{\ell_{\rho_i}(\alpha)}{2}=\ell_{\rho_i}(\gamma_0)=L.$$
Using the fact that the left-hand side term is bigger than $2\log\cosh\frac{\ell_
{\rho_i}(\alpha_j)}{2}+2\log\cosh\frac{\ell_
{\rho_i}(\alpha_i)}{2}$ for any $j\neq i$ and the assumption on $\ell_{\rho_i}(\alpha_i)$ we can deduce that for any $j\neq i$
$$\ell_{\rho_i}(\alpha_j)>\ell_{\rho_i}(\alpha_i)=\ell_{\rho_j}(\alpha_j)$$
so $$\min_i \ell_{\rho_i}(\alpha_j)\neq \ell_{\rho_j}(\alpha_j).$$

Look now at each term $\log\coth\frac{\ellF(\alpha)}{n}$: by convexity and monotonicity of $\log\coth(x)$, we have
\begin{align*}
\log\coth\frac{\ellF(\alpha)}{n}&=\log\coth\sum_{i=1}^n \frac{\ell_{\rho_i}(\alpha)}{2n}\leq\frac{1}{n}\sum_{i=1}^n\log\coth\frac{\ell_{\rho_i}(\alpha)}{2}\\
&\leq \max_i \log\coth\frac{\ell_{\rho_i}(\alpha)}{2}=\log\coth\min_i\frac{\ell_{\rho_i}(\alpha)}{2}.
\end{align*}
If $\alpha=\alpha_j$, we know thus that
$$\log\coth\frac{\ellF(\alpha_j)}{n}\leq \sum_{i\neq j}\log\coth \frac{\ell_{\rho_i}(\alpha)}{2}$$
while if $\alpha\neq \alpha_j$ for all $j$ we just consider that
$$\log\coth\frac{\ellF(\alpha)}{n}\leq \sum_{i=1}^n\log\coth \frac{\ell_{\rho_i}(\alpha)}{2}.$$
Using these inequalities and reordering the terms we get
$$n\sum_{\alpha\in\ort_{\Sigma}(\g_0)}\log\coth\frac{\ellF(\alpha)}{n}\leq n\sum_{i=1}^n\left(\sum_{\alpha\neq\alpha_i}\log\coth\frac{\ell_{\rho_i}(\alpha)}{2}\right).$$
On the other hand, the classical Basmajian's identity tells us that 
$$\sum_{\alpha\neq\alpha_i}2\log\coth\frac{\ell_{\rho_i}(\alpha)}{2}\leq \ell(\gamma_0)-2\log\coth\frac{\ell_{\rho_i}(\alpha_i)}{2}=\varepsilon.$$
and this implies that
$$2\sqrt n\sum_{\alpha\in\ort_\Sigma(\gamma_0)}\log \coth \frac{\ellR(\alpha)}{2\sqrt{n}}\leq 2n\sum_{\alpha\in\ort_{\Sigma}(\g_0)}\log\coth\frac{\ellF(\alpha)}{n}\leq n^2\varepsilon<\eta,$$
where the first inequality follows from Lemma \ref{lem:relRF}.

\end{proof}

\subsection{Proof of Corollary  \ref{corB}}
We conclude the paper with a proof of the geometric consequence of Theorem \ref{thm:main} mentioned in the introduction:
\begin{proof}[Proof of Corollary \ref{corB}]
We can assume that $\gamma$ is a boundary component, otherwise we cut the surface along the corresponding curve and consider the restriction of the representation to the fundamental groups of the surface(s) obtained. Note that Theorem \ref{Riemannian-component} implies that, for any orthotube $\alpha\in \ort_\Sigma(\gamma)$,
$$\ellR(\alpha)>2\sqrt{n}\arccoth\left(\exp\left(\frac{\ellR(\g)}{2\sqrt{n}}\right)\right)=:2w(\g).$$
Suppose, by contradiction, that the neighborhood $C(\g)=\langle \rho(\gamma)\rangle\backslash\mathcal N_{w(\gamma)}(\Yy_\gamma)$ is not embedded. This implies that there is an element $\delta\in\Gamma\setminus \<\g\>$ such that $C(\g)\cap\delta C(\g)\neq\emptyset$. 
Looking at the universal cover $\Xx$, this means that
$$\dR(\Yy_\g,\Yy_{\delta\gamma\delta^{-1}})<2w(\g)$$
and hence that there is an orthotube of length less than $2w(\g)$, a contradiction.

Similarly, consider another element $\delta$ corresponding to a simple closed curve or boundary component disjoint from the curve represented by $\g$. If by contradiction the two neighborhoods intersect, we get that
$$\dR(\Yy_\g, \Yy_\delta)<w(\g)+w(\delta)\leq 2\max{w(\g),w(\delta)}$$
so we have again an orthogeodesic (seen as element of $\ort_\g(\Sigma)$ or $\ort_\delta(\Sigma)$, depending on whether $w(\g)$ or $w(\delta)$ is maximum) which is too short, a contradiction.
\end{proof}
\bibliographystyle{alpha}
\bibliography{references_basmajian}
\end{document}